\documentclass[10pt,a4paper]{article}
\usepackage{fullpage,mathrsfs,graphicx,framed,color,amssymb,amsmath,amsthm}
\usepackage[boxed, noline, ruled, linesnumbered]{algorithm2e}
\usepackage{epsfig, graphicx}
\usepackage{latexsym,amsfonts,amsbsy,amssymb}
\usepackage{amsmath,amsthm}
\usepackage{enumerate}
\usepackage{booktabs}  
\usepackage{multirow}   
\makeatletter 

\@addtoreset{equation}{section}
\makeatother 
\allowdisplaybreaks 
\newtheorem{theorem}{Theorem}[section]
\newtheorem{lemma}{Lemma}[section]
\newtheorem{corollary}{Corollary}[section]
\newtheorem{remark}{Remark}[section]
\newtheorem{definition}{Definition}[section]

\begin{document}
\title{Mini mixed finite element method for nearly incompressible linear elasticity problems}
\author{Zhijin Guan\footnote{School of Science, East China University of Technology, Nanchang 330013, China.
({\tt guanzhijin@ecut.edu.cn})},\ \ \
Yue Feng\footnote{School of Mathematics, Jilin University,
Changchun 130012, China. ({\tt fengyue\_math@jlu.edu.cn})},\ \ \
Hehu Xie\footnote{SKLMS, ICMSEC,
Academy of Mathematics and Systems Science, Chinese Academy of
Sciences,  Beijing 100190, China, and School of Mathematical
Sciences, University of Chinese Academy of Sciences, Beijing 100049, China.
({\tt hhxie@lsec.cc.ac.cn})} \  \
and\  \ Chenguang Zhou\footnote{Department of Mathematics, Beijing University of Technology, Beijing 100124, China. (Corresponding author: {\tt zhoucg@bjut.edu.cn})} }
\date{}
\maketitle
\begin{abstract}
This paper addresses the numerical solution of nearly incompressible linear elasticity boundary value problems and their associated eigenvalue problems.
A mixed finite element formulation based on Mini element is proposed to circumvent the locking phenomenon that plagues standard
low-order elements in the nearly incompressible limit. For the boundary value problem, we establish the well-posedness
of mixed variational formulation and derive a priori error estimates that are uniform with respect to the Lam\'{e} constant $\underline{\lambda}$,
thereby proving the method's locking-free property. For the eigenvalue problem, we develop an efficient non-nested augmented
subspace algorithm designed within the mixed finite element framework. A comprehensive convergence analysis is provided
for the discrete eigenvalue approximation and the proposed iterative solver, demonstrating that the convergence rates
remain independent of $\underline{\lambda}$. Numerical experiments on both problems confirm the theoretical conclusions,
showing optimal convergence rates and robustness as $\underline{\lambda} \to \infty$. The results validate the effectiveness
of Mini element and proposed augmented subspace algorithm for reliable and efficient computation in the nearly incompressible regime.

\end{abstract}

\section{Introduction}
The study of deformation behavior of nearly incompressible elastic materials has significant application value in numerous important physical problems.
The concept of nearly incompressible materials originated from their intermediate step characteristics in the solution process of plane elastic problems,
which leads to difficulties in numerical simulations for this type of material.
Linear elasticity theory investigates the deformation rules and internal stress distribution of solid structures
under given load conditions and is widely applied in structural analysis and engineering design fields.
Given the broad application background of elasticity problems, the academic community is increasing paying attention to their mathematical models and theoretical analyses.

Let $\Omega \subset \mathbb{R}^d \, (d=2,3)$ be a bounded convex domain with Lipschitz continuous boundary $\partial \Omega$.
We consider the pure displacement boundary problem of linear elasticity: Find the displacement field $\mathbf{u}$ such that
\begin{equation}\label{eq:ElasticStrongForm}
\begin{cases}
-\underline{\mu}\Delta \mathbf{u} - (\underline{\lambda} + \underline{\mu}) \nabla (\mathrm{div} \mathbf{u}) = \mathbf{f}, & \text{in } \Omega, \\
\mathbf{u} = \mathbf{0}, & \text{on } \partial \Omega,
\end{cases}
\end{equation}
where $\mathbf{f}$ represents the body force. The Lam\'{e} constants $\underline{\lambda}$ and $\underline{\mu}$ are material parameters determined
by the Poisson ratio $\nu \in (0, 0.5)$ and Young's modulus $E$, expressed as
\begin{align*}
\underline{\mu} := \frac{E}{2(1+\nu)}, \quad
\underline{\lambda} := \frac{E\nu}{(1+\nu)(1-2\nu)} \in (0, +\infty),
\end{align*}
where $\underline{\mu} \in [\mu_1, \mu_2]$ with $0 < \mu_1 < \mu_2 \ll \infty$.

A fundamental computational challenge arises in the nearly incompressible limit, where $\nu\rightarrow 0.5$ and consequently $\underline\lambda \rightarrow \infty$ and $\nabla\cdot\mathbf{u}\rightarrow 0$.
At this point, the solutions obtained by standard low-order conforming finite
element methods typically fail to approximate the true solution of original problem,
or the convergence rate cannot reach the optimal order \cite{MR2373954,MR0520174}.
This phenomenon is known as the locking phenomenon.
To overcome this problem, extensive research has been conducted, leading to
various approaches including enriched Galerkin method \cite{MR4972268, MR4798315,MR4358467}, hybrid high-order method \cite{MR4891729,MR3283758}, discontinuous Galerkin method \cite{MR2219020,MR2997905}, virtual element method \cite{MR4808777,MR4119968}, weak Galerkin method \cite{MR4780316, MR4710941}, et cetera.

Mini element, originally introduced for Stokes flow \cite{MR0520174}, is a celebrated low-order mixed element that fulfills the inf-sup condition.
Its application to linear elasticity, while natural, requires a careful analysis to ensure robustness as $\underline\lambda \rightarrow \infty$.
Although its use in elasticity is known, a comprehensive a priori error analysis establishing uniform convergence for both the boundary value
problem and its associated eigenvalue problem remains a topic worthy of detailed exposition and verification.

In view of the numerical solution difficulties for nearly incompressible linear elasticity problems, in next section,
we discuss the performance of mixed variational formulation based on the Mini element in overcoming the locking phenomenon.
To this end, by introducing the pressure variable $p := (\lambda + \mu) \nabla\cdot\mathbf{u}$, we obtain the corresponding
mixed formulation for linear elasticity problem \eqref{eq:ElasticStrongForm}: Find $(\mathbf u,p)$ such that
\begin{align*}
\begin{cases}
- \underline{\mu} \Delta \mathbf{u} - \nabla p = \mathbf{f}, & \text{in}~\Omega, \\
\nabla \cdot \mathbf{u} - (\underline{\lambda} + \underline{\mu})^{-1}p = 0,  &\text{in}  ~\Omega, \\
\mathbf{u} = \mathbf{0}, &\text{on} \, \partial\Omega.
\end{cases}
\end{align*}

Beyond boundary value problems, the elasticity eigenvalue problem is fundamental for analyzing vibration modes and stability of structures \cite{MR2652780}.
However, the development of efficient and locking-free numerical solvers for this eigenvalue problem has received comparatively
less attention than the corresponding source problem.
Standard iterative eigensolvers \cite{Bai,Chatelin,Sorensen} applied to the locked discrete system can perform poorly.
While multilevel and subspace correction methods for eigenvalue problems have seen considerable advances \cite{full,HongXieXu,LinXie_MultiLevel,Xie_JCP,XuXieZhang,XuZhou},
their adaptation to the mixed formulation of nearly incompressible elasticity, ensuring uniform convergence in $\underline\lambda$,
presents distinct challenges. Therefore, in this paper, we also study the efficient numerical algorithms for the following mixed eigenvalue problem of elasticity:
Find $(\lambda, \mathbf u,p)$ 
such that
\begin{align*}
\begin{cases}
- \underline{\mu} \Delta \mathbf{u} - \nabla p = \lambda \mathbf{u}, & \text{in}~\Omega, \\
\nabla \cdot \mathbf{u} - (\underline{\lambda} + \underline{\mu})^{-1}p = 0,  &\text{in} ~\Omega, \\
\mathbf{u} = \mathbf{0}, &\text{on} \, \partial\Omega.
\end{cases}
\end{align*}

Efficient numerical methods for eigenvalue problems often build upon fast solvers for the corresponding source problems.
Multigrid and multilevel methods are among the most efficient algorithms for elliptic boundary value problems, offering optimal computational complexity \cite{MR2373954}. Their extension to eigenvalue problems has led to fruitful developments, such as two-grid methods \cite{XuZhou},
schemes combining shift-inverse techniques \cite{HuCheng,YangBi}, and innovative multilevel correction methods that transform the solution of eigenvalue problems
into a series of linear boundary value problems on a hierarchy of spaces and a small-scale eigenvalue problem on a carefully designed correction space  \cite{Dang,HongXieXu,LinXie_MultiLevel,Xie_JCP,Xie_IMA,MR3434038,XieZhangOwhadi,XuXieZhang,MR4512632}.  These ideas inspire the design of novel
solvers for the challenging elasticity eigenvalue problem.

To sum up, the first aim of this paper is to establish the theory of Mini mixed finite element method for solving nearly incompressible
elastic boundary value problem and its corresponding eigenvalue problem.
Based on the error estimates for Mini mixed finite element method,
a type of augmented subspace algorithm is developed for solving the nearly incompressible eigenvalue problems in the non-nested mixed finite element spaces,
which is the second aim here.
The theoretical analysis here show that the Mini mixed finite element method has the uniform convergence order according to the Lam\'{e} constant $\underline{\lambda}\rightarrow \infty$.
Furthermore, the proposed augmented subspace method remains the uniform convergence for solving the nearly incompressible eigenvalue problem,
which is also validated by the numerical examples of this paper.

The remainder of this paper is organized as follows. Section \ref{Section_2} details the mixed variational formulation, well-posedness analysis and locking-free error estimates for the boundary value problem. Section \ref{sec:EigenAnalysis} extends the analysis to the eigenvalue problem, and in Section \ref{section4}, we introduce the augmented subspace algorithm and its convergence theory. In Section \ref{section5}, numerical experiments are conducted to verify the effectiveness and efficiency of our proposed methods. Finally, in Section \ref{section6}, we provide the concluding remarks.

\section{Mini mixed finite element discretization for linear elasticity boundary value problems}\label{Section_2}
In this section, we establish the theory of subspace approximation for the nearly incompressible linear elasticity
boundary value problem in the mixed form. As one of the applications,
the error estimate for Mini mixed finite element method
is derived. In addition, it should be mentioned that the letter $C$, either with or without subscripts, represents a general positive
constant that may vary depending on where it appears in this paper.

\subsection{Well-posedness of mixed variational form}
This subsection is to build the well-posedness theory for mixed variational problems,
including the existence and uniqueness of solution, as well as parameter-robust stability analysis.

In this paper, the standard Sobolev space $H^{l}(\Omega)$ with $l\geq 0$ in \cite{MR0450957} is used. The associated norm in $H^{l}(\Omega)$ is denoted by $\|\cdot\|_l$, which may also be used to represent the vector norm in $\left(H^{l}(\Omega)\right)^d$, provided that there is no ambiguity. In the case where $l=0$, the space of square-integrable functions $L^2(\Omega)$ corresponds with $H^0(\Omega)$. The notation $H_0^1(\Omega)$ represents the subspace of $ H^1(\Omega)$ that contains the functions on $\partial\Omega$ with vanishing traces. $L^2_0(\Omega)$ is denoted by a subspace of $L^2(\Omega)$ and the integral of the elements within $\Omega$ is zero. For the sake of simplicity, we set $\mathbf{V} := \left(H_0^1(\Omega)\right)^d$ and $Q := L^2_0(\Omega)$.
Correspondingly, we denote the vector norm $\|\cdot\|_1$ by $\|\cdot\|_{\mathbf V}$ and the scalar norm $\|\cdot\|_0$ by $\|\cdot\|_{Q}$. A superscript prime is used to denote the dual space, with the corresponding norm being the dual norm.
Given source terms $\mathbf{f} \in \mathbf{V}^{\prime}$ and $g \in Q^{\prime}$, where $\mathbf{V}^{\prime}$ and $Q^{\prime}$ are respectively the dual spaces of $\mathbf{V}$ and $Q$,
consider the following mixed variational problem:
Find $(\mathbf{u}, p) \in \mathbf{V} \times Q$ such that
\begin{equation}\label{eq:MixWeakFormGeneral}
\begin{cases}
a(\mathbf{u}, \mathbf{v}) + b(\mathbf{v}, p) = \mathbf{f}(\mathbf{v}), & \forall \mathbf{v} \in \mathbf{V}, \\
b(\mathbf{u}, q) - c(p, q) = g(q), & \forall q \in Q,
\end{cases}
\end{equation}
where $\mathbf{f}(\mathbf{v}) := \int_{\Omega} \mathbf{f} \cdot \mathbf{v} \, \mathrm{d}{\bf x}$, $g(q) := \int_{\Omega} g q \, \mathrm{d}{\bf x}$ and the bilinear forms are defined as
\begin{align*}
a(\mathbf{u}, \mathbf{v}) &:= \underline{\mu} \int_{\Omega} \nabla \mathbf{u} \cdot \nabla \mathbf{v} \, \mathrm{d}{\bf x}, \\
b(\mathbf{v}, q) &:= \int_{\Omega} (\mathrm{div} \mathbf{v}) q \, \mathrm{d}{\bf x}, \\
c(p, q) &:= \frac{1}{\underline{\lambda} + \underline{\mu}} \int_{\Omega} p q \, \mathrm{d}{\bf x}.
\end{align*}

Next, based on the definitions of bilinear forms $a(\cdot,\cdot)$, $b(\cdot,\cdot)$
and $c(\cdot,\cdot)$, we derive several key conclusions.
It is easy to prove that $a(\cdot,\cdot)$ is coercive and bounded, i.e.,
\begin{equation}\label{eq:abilinear}
a(\mathbf{v}, \mathbf{v}) \geq \alpha \|\mathbf{v}\|_{\mathbf{V}}^{2},
\quad a(\mathbf{u}, \mathbf{v}) \leq C_a \|\mathbf{u}\|_{\mathbf{V}}
\|\mathbf{v}\|_{\mathbf{V}}, \quad \forall \mathbf{u}, \mathbf{v} \in \mathbf{V},
\end{equation}
where $C_a > 0$ and $\alpha > 0$ are independent of $\underline{\lambda}$.
There exist constants $C_b > 0$ and $\beta > 0$
independent of $\underline{\lambda}$, such that (cf. \cite{MixedHybrid})
\begin{equation}\label{eq:bbilinear}
\inf_{q \in Q} \sup_{\mathbf{v} \in \mathbf{V}} \frac{b(\mathbf{v}, q)}{\|\mathbf{v}\|_{\mathbf{V}} \|q\|_{Q}} \geq \beta, \qquad
b(\mathbf{v}, q) \leq C_b \|\mathbf{v}\|_{\mathbf{V}} \|q\|_{Q}, \quad \forall \mathbf{v} \in \mathbf{V}, \, \forall q \in Q.
\end{equation}
And for all $p, q \in Q$, it is easy to know that the following formula holds
\begin{equation}\label{eq:cbilinear}
c(p, q) \leq \frac{1}{\underline{\lambda} + \underline{\mu}} \|p\|_{Q} \|q\|_{Q}.
\end{equation}

To make the exposition more concise, we define a product space as follows.
\begin{definition}
Let $\mathcal{X} := \mathbf{V} \times Q$
be the Cartesian product space of $\mathbf{V}$ and $Q$. For any functions
$\Phi = (\mathbf{u}, p) \in \mathcal{X}$ and $\Psi = (\mathbf{v}, q) \in \mathcal{X}$,
the inner product is defined as
$$(\Phi, \Psi)_{\mathcal{X}} := (\mathbf{u}, \mathbf{v})_{\mathbf{V}} + (p, q)_{Q},$$
where $(\cdot, \cdot)_{\mathbf{V}}$ is the inner product on $\mathbf{V}$,
and $(p, q)_{Q}$ is the $L^2(\Omega)$ inner product. The induced norm is
\begin{equation}\label{ProductSpaceNorm}
\|\Psi\|_{\mathcal{X}} := \sqrt{\|\mathbf{v}\|_{\mathbf{V}}^{2} + \| q\|_{Q}^{2}},
\end{equation}
where $\|\cdot\|_{\mathbf{V}}$ and $\|\cdot\|_{Q}$ are the norms
induced by the inner products on $\mathbf{V}$ and $Q$, respectively.
It is easy to know that the space $\mathcal{X}$ is a Hilbert space.
\end{definition}

\begin{theorem}\label{thm:Wellposedness}
Given source terms $\mathbf{f} \in \mathbf{V}^{\prime}$ and $g \in Q^{\prime}$,
the mixed variational problem \eqref{eq:MixWeakFormGeneral}
admits a unique solution $(\mathbf{u}, p) \in \mathbf{V} \times Q$,
which satisfies the following uniform bound
\begin{equation}\label{Stability_Inequality}
\|\mathbf{u}\|_{\mathbf{V}} + \|p\|_{Q} \leq C_{\rm stab}
\left( \|\mathbf{f}\|_{\mathbf{V}^{\prime}} + \|g\|_{Q^{\prime}} \right),
\end{equation}
where the constant $C_{\rm stab} > 0$ is independent of the Lam\'{e} constant $\underline{\lambda}$.
\end{theorem}
\begin{proof}
For any fixed $\underline{\lambda}\in \mathbb R_+$, let $\Phi = (\mathbf{u}, p)\in \mathcal X$
and $\Psi = (\mathbf{v}, q) \in \mathcal{X}$. Define
\begin{align}
A(\Phi, \Psi) &:= a(\mathbf{u}, \mathbf{v}) + b(\mathbf{v}, p)
+ b(\mathbf{u}, q) - c(p, q), \label{A_define} \\
F(\Psi) &:= \mathbf{f}(\mathbf{v}) + g(q). \label{F_define}
\end{align}
Then, the equation \eqref{eq:MixWeakFormGeneral} can  be written as:
Find $\Phi \in \mathcal{X}$ such that
\begin{equation}
A(\Phi, \Psi) = F(\Psi), \quad \forall \Psi \in \mathcal{X}.
\end{equation}

i) Boundedness: Let us define the constant $C_{\rm tmp} := \max\{ C_a, C_b, 1/\underline{\mu} \}$,
which is bounded and independent of $\underline{\lambda}$.
Then, from equations \eqref{eq:abilinear}, \eqref{eq:bbilinear}, \eqref{eq:cbilinear}, \eqref{ProductSpaceNorm}, \eqref{A_define} and the triangle inequality,
we can obtain 
\begin{align*}
&|A(\Phi,\Psi)| \leq |a(\mathbf{u}, \mathbf{v})| + |b(\mathbf{v}, p)|
+ |b(\mathbf{u}, q)| + |c(p,q)| \\
&\leq  C_a \|\mathbf{u}\|_{\mathbf{V}} \|\mathbf{v}\|_{\mathbf{V}}
+ C_b \|\mathbf{v}\|_{\mathbf{V}} \|p\|_{Q}
+ C_b \|\mathbf{u}\|_{\mathbf{V}} \|q\|_{Q} + \frac{1}{\underline{\lambda}
+ \underline{\mu}}\|p\|_{Q} \|q\|_{Q}  \\
&\leq C_{\rm tmp} \left( \|\mathbf{u}\|_{\mathbf{V}} + \|p\|_{Q} \right)
\left(  \|\mathbf{v}\|_{\mathbf{V}} + \|q\|_{Q} \right) \\
&\leq 2  C_{\rm tmp} \sqrt{ \|\mathbf{u}\|_{\mathbf{V}}^2
+ \|p\|_{Q}^2 } \sqrt{\|\mathbf{v}\|_{\mathbf{V}}^2 + \|q\|_{Q}^2} \\
&= 2 C_{\rm tmp} \|\Phi\|_{\mathcal{X}} \|\Psi\|_{\mathcal{X}}.
\end{align*}
Similarly,  it is easy to know that the following inequalities hold
\begin{align*}
&|F(\Psi)| \leq |\mathbf{f}(\mathbf{v})| + |g(q)|
\leq  \|\mathbf{f}\|_{\mathbf{V}^{\prime}} \|\mathbf{v}\|_{\mathbf{V}}
+ \|g\|_{Q^{\prime}} \|q\|_{Q} \\
&\leq ( \|\mathbf{f}\|_{\mathbf{V}^{\prime}} + \|g\|_{Q^{\prime}} )
\left(  \|\mathbf{v}\|_{\mathbf{V}} + \|q\|_{Q} \right) \\
&\leq \sqrt{2} ( \|\mathbf{f}\|_{\mathbf{V}^{\prime}}
+ \|g\|_{Q^{\prime}} ) \sqrt{ \|\mathbf{v}\|_{\mathbf{V}}^2 + \|q\|_{Q}^2 } \\
&= \sqrt{2} ( \|\mathbf{f}\|_{\mathbf{V}^{\prime}}
+ \|g\|_{Q^{\prime}} ) \|\Psi\|_{\mathcal{X}}.
\end{align*}

ii) Inf-sup Condition: For any $\Phi = (\mathbf{u}, p)\in \mathcal{X}$,
there exist a positive constant $C_p$ and $\mathbf{v}_p\in V$ such that
\begin{equation}\label{tmpinfsupcollary}
\mathrm{div} \mathbf{v}_{p} = p, \quad
\|\mathbf{v}_{p}\|_{\mathbf{V}} \leq C_p \|p\|_{Q}, \quad
b(\mathbf{v}_p, p) \geq \beta \|\mathbf{v}_p\|_{\mathbf{V}} \|p\|_Q.
\end{equation}
Then let us set $\Psi = (\mathbf{u} + \gamma \mathbf{v}_p, -p)$,
where the positive real number $\gamma$ satisfies
$$\frac{\sqrt{2} \gamma \beta}{2} - \frac{\gamma^2C_a^2 C_p^2}{2\alpha} > 0.$$
Combing equations \eqref{eq:abilinear}, \eqref{eq:bbilinear}, \eqref{tmpinfsupcollary} and
Young's inequality leads to
\begin{align}\label{FZscaling}
&A(\Phi, \Psi) = a(\mathbf{u}, \mathbf{v}) + b(\mathbf{v}, p)
+ b(\mathbf{u}, q) - c(p,q) \notag \\
&= a(\mathbf{u}, \mathbf{u}) + \gamma a(\mathbf{u}, \mathbf{v}_p)
+ b(\mathbf{u}, p) + \gamma b(\mathbf{v}_p, p) - b(\mathbf{u}, p) + c(p, p) \notag \\
&\geq \alpha \|\mathbf{u}\|_{\mathbf{V}}^2
- \gamma C_a \|\mathbf{u}\|_{\mathbf{V}} \|\mathbf{v}_p\|_{\mathbf{V}}
+ \gamma \beta \|\mathbf{v}_p\|_{\mathbf{V}} \|p\|_Q
+ \frac{1}{\underline{\lambda} + \underline{\mu}} \|p\|_Q^2 \notag \\
&\geq \alpha \|\mathbf{u}\|_{\mathbf{V}}^2 - \frac{\alpha}{2} \|\mathbf{u}\|_{\mathbf{V}}^2
- \frac{\gamma^2C_a^2}{2\alpha} \|\mathbf{v}_p\|_{\mathbf{V}}^2
+ \gamma \beta \|\mathbf{v}_p\|_{\mathbf{V}} \|p\|_Q
+ \frac{1}{\underline{\lambda} + \underline{\mu}} \|p\|_Q^2 \notag \\
&\geq \alpha \|\mathbf{u}\|_{\mathbf{V}}^2 - \frac{\alpha}{2} \|\mathbf{u}\|_{\mathbf{V}}^2
- \frac{\gamma^2C_a^2 C_p^2}{2\alpha} \|p\|_{Q}^2
+ \frac{\sqrt{2} \gamma \beta}{2} \|p\|_Q^2 + \frac{1}{\underline{\lambda}
+ \underline{\mu}} \|p\|_Q^2 \notag \\
&\geq \frac{\alpha}{2} \|\mathbf{u}\|_{\mathbf{V}}^2
+ \left(\frac{\sqrt{2} \gamma \beta}{2}
- \frac{\gamma^2C_a^2 C_p^2}{2\alpha}\right) \|p\|_Q^2 \nonumber\\
&\geq  C_{\rm A} (\|\mathbf{u}\|_{\mathbf{V}}^2 + \|p\|_Q^2)
= C_{\rm A} \|\Phi\|_{\mathcal{X}}^2,
\end{align}
where $C_{\rm A} = \min\left\{ \alpha/2,
\sqrt{2} \gamma \beta/2 - \gamma^2C_a^2 C_p^2/(2\alpha) \right\}$.

From equation \eqref{tmpinfsupcollary} and the triangle inequality, we obtain
\begin{align}\label{FMscaling}
&\|\Psi\|_{\mathcal{X}} = \sqrt{\|\mathbf{v}\|_{\mathbf{V}}^2 + \|q\|_Q^2}
= \sqrt{(\|\mathbf{u}\|_{\mathbf{V}}
+ \gamma \|\mathbf{v}_p\|_{\mathbf{V}})^2 + \|p\|_Q^2}  \nonumber \\
&\leq\sqrt{2 \|\mathbf{u}\|_{\mathbf{V}}^2 + 2\gamma^2 \|\mathbf{v}_p\|_{\mathbf{V}}^2
+ \|p\|_Q^2}  \nonumber \\
&\leq\sqrt{2 \|\mathbf{u}\|_{\mathbf{V}}^2 + (1 + 2\gamma^2C_p^2) \|p\|_Q^2}  \nonumber \\
&\leq C_{\rm B} \sqrt{\|\mathbf{u}\|_{\mathbf{V}}^2 + \|p\|_Q^2}
= C_{\rm B} \|\Phi\|_{\mathcal{X}},
\end{align}
where $C_{\rm B} = \max\left\{ \sqrt{2}, \sqrt{1 + 2\gamma^2C_p^2} \right\}.$

Then from equations \eqref{FZscaling} and \eqref{FMscaling}, the following inequalities hold
\begin{align}\label{inf-sup_gengral}
\sup_{\Psi \in \mathcal{X}} \frac{A(\Phi, \Psi)}{\|\Psi\|_{\mathcal{X}}}
&\geq \frac{A((\mathbf{u}, p), (\mathbf{u} - \gamma \mathbf{v}_p, -p))}
{\|(\mathbf{u} - \gamma \mathbf{v}_p, -p)\|_{\mathcal{X}}}
\geq \tilde{\beta} \|\Phi\|_{\mathcal{X}},
\end{align}
where $\tilde{\beta} = C_{\rm A}/C_{\rm B}$,
noting that $\tilde{\beta}$ is independent of the Lam\'{e} constant $\underline{\lambda}$.

iii) Symmetry: The bilinear form $A(\cdot, \cdot)$ is symmetric, ensuring non-degeneracy.
Combining this with the estimate \eqref{inf-sup_gengral}, we derive
\begin{equation}
\sup_{\Phi \in \mathcal{X}} \frac{A(\Phi, \Psi)}{\|\Phi\|_{\mathcal{X}}}
= \sup_{\Phi \in \mathcal{X}} \frac{A(\Psi, \Phi)}{\|\Phi\|_{\mathcal{X}}}
\geq 0, \quad \forall \mathbf{0} \neq \Psi \in \mathcal{X}.
\end{equation}

By virtue of the Babu\v{s}ka-Lax-Milgram theorem \cite{MR2373954} and i)--iii),
the mixed variational problem  \eqref{eq:MixWeakFormGeneral} admits a unique solution satisfying
$$\|\Phi\|_{\mathcal{X}} \leq \frac{1}{\tilde{\beta}} \|F\|_{\mathcal{X}^{\prime}},$$
where $\tilde{\beta} = C_{\rm A}/C_{\rm B}$ is independent of $\underline{\lambda}$.
This yields the stability estimate (\ref{Stability_Inequality}).
\end{proof}

\begin{remark}
Using the Stokes regularity theory \cite{MixedHybrid},
the solution of \eqref{eq:MixWeakFormGeneral} has the following regularity result
$$\|\mathbf{u}\|_{1+s} + \|p\|_{s}
\leq C \left( \| \mathbf{f}\|_{0} + \| g\|_{0} \right),$$
where $s \in (0, 1]$ depends on the domain geometry and $C > 0$ is independent of $\underline{\lambda}$.
\end{remark}

\begin{remark}
As a direct corollary, when $g = 0$, the equation \eqref{eq:MixWeakFormGeneral}
reduces to the mixed formulation for linear elasticity problem.
In this case, the stability estimate reduces to
$$\|\mathbf{u}\|_{\mathbf{V}} + \| p\|_{Q} \leq C_{\rm stab} \|\mathbf{f}\|_{\mathbf{V}^{\prime}},$$
where the constant $C_{\rm stab} > 0$ is independent of $\underline{\lambda}$.
This result demonstrates the robustness of the mixed formulation
for linear elasticity problem in the nearly incompressible limit $\underline{\lambda} \to \infty$.
Furthermore, we have the regularity result
\begin{equation}
\|\mathbf{u}\|_{1+s} + \|p\|_{s} \leq C \|\mathbf{f}\|_{0},
\end{equation}
where $s\in(0,1]$ depends on the geometry of the spatial domain $\Omega$.
\end{remark}

\subsection{Mini mixed finite element method and a priori error estimates}
This subsection is devoted to establishing the convergence theory
for mixed finite element discretization and provides a locking-free priori error estimate.
Let $\mathcal{T}_h$ be a quasi-uniform triangulation of the computing domain $\Omega$. The diameter of a cell $K\in \mathcal{T}_h$ is denoted by $h_K$ and the mesh size $h$ describes the maximal diameter of all cells $K\in \mathcal{T}_h$. The notations $\mathbf{V}_{h}$ and $Q_{h}$ are denoted by the finite-dimensional subspaces of $\mathbf{V}$ and $Q$, respectively. Define $\mathcal{X}_h := \mathbf{V}_h \times Q_{h}$,
which constitutes a finite-dimensional subspace of $\mathcal{X}$.
Inheriting the inner product from $\mathcal{X}$, $\mathcal{X}_h$ remains a Hilbert space.

Based on the subspace $\mathcal X_h$, we can define the corresponding
numerical scheme to the equation \eqref{eq:MixWeakFormGeneral}:
Find $(\mathbf{u}_{h}, p_{h}) \in \mathcal{X}_{h}$ such that
\begin{equation}\label{eq:MixDiscreteWeakFormGeneral}
\begin{cases}
a(\mathbf{u}_{h}, \mathbf{v}_{h}) + b(\mathbf{v}_{h}, p_{h}) = \mathbf{f}(\mathbf{v}_{h}),
& \forall \mathbf{v}_{h} \in \mathbf{V}_{h}, \\
b(\mathbf{u}_{h}, q_{h}) - c(p_{h}, q_{h}) = g(q_{h}), & \forall q_{h} \in Q_{h}.
\end{cases}
\end{equation}
Since the finite dimensional space $\mathcal X_h$ is a subspace of $\mathcal X$, the stability estimate
(\ref{Stability_Inequality}) can be directly used to deduce the error estimate for
the numerical solution $(\mathbf u_h, p_h)$ defined by (\ref{eq:MixDiscreteWeakFormGeneral}).
\begin{theorem}\label{thm:errorEstimate}
Assume that there exists a bounded linear operator
$\pi_{h}: \mathbf{V} \rightarrow \mathbf{V}_{h}$ such that
\begin{enumerate}
\item[(1)] $b\left(\mathbf{v} - \pi_{h} \mathbf{v}, q_h\right) = 0, \quad \forall q_h \in Q_{h},\
\forall \mathbf v\in \mathbf V$,
\item[(2)] $\left\|\pi_{h} \mathbf{v}\right\|_{\mathbf{V}} \leq C_{\pi}\| \mathbf{v}\|_{\mathbf{V}},
\quad \forall \mathbf{v} \in \mathbf{V},$
\end{enumerate}
where $C_{\pi}$ is independent of the mesh size $h$.
Then, the approximate solution of \eqref{eq:MixDiscreteWeakFormGeneral} satisfies the following inequality,
\begin{eqnarray}\label{eq:ErrorEstimate}
\|\mathbf{u} - \mathbf{u}_{h}\|_{\mathbf{V}} + \|p - p_{h}\|_{Q} \leq C_{\rm err}
\left( \inf_{\mathbf{v}_{h} \in \mathbf{V}_{h}}
\| \mathbf{u} - \mathbf{v}_{h}\|_{\mathbf{V}}
+ \inf_{q_{h} \in Q_{h}} \| p - q_{h}\|_{Q} \right),
\end{eqnarray}
where $C_{\rm err} > 0$ is independent of $\underline{\lambda}$ and the mesh size $h$.
\end{theorem}
\begin{proof}
Subtracting \eqref{eq:MixDiscreteWeakFormGeneral} from \eqref{eq:MixWeakFormGeneral},
for any $(\tilde{\mathbf{u}}_{h}, \tilde{p}_{h}) \in \mathbf{V}_h \times Q_{h}$,
we obtain the error equations
\begin{equation}\label{eq:MixDiscreteWeakFormGeneral4Error}
\begin{cases}
a(\tilde{\mathbf{u}}_{h} - \mathbf{u}_{h}, \mathbf{v}_{h})
+ b(\mathbf{v}_{h}, \tilde{p}_{h} - p_{h})
= a( \tilde{\mathbf{u}}_{h} - \mathbf{u}, \mathbf{v}_{h})
+ b(\mathbf{v}_{h}, \tilde{p}_{h} - p), & \forall \mathbf{v}_{h} \in \mathbf{V}_{h},\\
b(\tilde{\mathbf{u}}_{h} - \mathbf{u}_{h}, q_{h}) - c(\tilde{p}_{h} - p_{h}, q_{h})
= b(\tilde{\mathbf{u}}_{h} - \mathbf{u}, q_{h}) - c(\tilde{p}_{h} - p, q_{h}),
& \forall q_{h} \in Q_{h}.
\end{cases}
\end{equation}
Due to the Fortin's criterion \cite{MR2373954} and conditions of the operator $\pi_h$,
the following inf-sup condition holds
\begin{eqnarray}\label{inf_sup_bh}
\inf_{q_h \in Q_h} \sup_{\mathbf{v}_h \in \mathbf{V}_h}
\frac{b(\mathbf{v}_h, q_h)}{\|\mathbf{v}_h\|_{\mathbf{V}} \|q_h\|_{Q}} \geq \frac{\beta}{C_\pi},
\end{eqnarray}
where the constant $\beta$ is defined in \eqref{eq:bbilinear}.

Then, from (\ref{eq:abilinear}), (\ref{eq:bbilinear}), (\ref{eq:cbilinear}),
(\ref{eq:MixDiscreteWeakFormGeneral4Error}) and (\ref{inf_sup_bh}),
using the same derivative process to Theorem \ref{thm:Wellposedness},
there exists the constant $C_{\rm err}>1$ such that
the discrete function $(\tilde{\mathbf{u}}_{h} - \mathbf{u}_{h}, \tilde{p}_{h} - p_{h})\in \mathcal X_h$
satisfies the stability estimate
\begin{align}\label{Inequality_1}
\|\tilde{\mathbf{u}}_{h} - \mathbf{u}_{h}\|_{\mathbf{V}}
+ \| \tilde{p}_{h} - p_{h}\|_{Q} \leq (C_{\rm err} -1)
\left( \| \tilde{\mathbf{u}}_{h} - \mathbf{u}\|_{\mathbf{V}} + \| \tilde{p}_{h} - p\|_{Q} \right),
\end{align}
where $C_{\rm err}$ is independent of $\underline{\lambda}$ and the mesh size $h$.
Combining (\ref{Inequality_1}) and the triangle inequality leads to the following estimate
\begin{eqnarray}\label{Inequality_2}
\|\mathbf u-\mathbf u_h\|_{\mathbf V}+\|p-p_h\|_{Q}
&\leq& \|\mathbf{u} - \tilde{\mathbf{u}}_{h} \|_{\mathbf{V}}
+ \| p - \tilde{p}_{h} \|_{Q}+\|\tilde{\mathbf{u}}_{h} - \mathbf{u}_{h}\|_{\mathbf{V}}
+ \| \tilde{p}_{h} - p_{h}\|_{Q}\nonumber \\
&\leq& C_{\rm err} \left( \| \mathbf{u} - \tilde{\mathbf{u}}_{h}\|_{\mathbf{V}}
+ \| p - \tilde{p}_{h}\|_{Q} \right).
\end{eqnarray}
The desired result (\ref{eq:ErrorEstimate}) can be deduced from (\ref{Inequality_2})
and the arbitrariness of $(\tilde{\mathbf{u}}_{h}, \tilde p_h)\in\mathcal X_h$.
Then the proof is completed.
\end{proof}

The locking-free property of mixed finite element discretization hinges on constructing a bounded linear operator $\pi_{h}$ satisfying Theorem \ref{thm:errorEstimate}. We illustrate this by the Mini element in two-dimension.



In order to introduce the Mini element, we first define the bubble function as follows,
\[
B_3 := \Big\{ b(\mathbf{x}) \bigm| b(\mathbf{x})|_K \in \mathcal P_3(K) \cap H_0^1(K),\
\forall K \in \mathcal{T}_h \Big\},
\]
where $\mathcal P_3(K)$ denotes the polynomial space of order three defined on the cell $K\subset \mathcal{T}_h$.
Following \cite{ArnoldBrezziFortin}, we set
\begin{eqnarray}\label{Mini_Mixed_FEM}
\mathbf V_h := \bigl\{ \mathcal{L}_1^1(\mathcal{T}_h) \oplus B_3 \bigr\}^2 \cap\mathbf V,\qquad
Q_h := \mathcal{L}_1^1(\mathcal{T}_h) \cap Q,
\end{eqnarray}
where $\mathcal L_1^1(\mathcal T_h)$ is the linear conforming finite element space
defined on the mesh $\mathcal T_h$.
In order to use the Fortin's law to deduce the stable and convergence results for
the Mini mixed finite element space (\ref{Mini_Mixed_FEM}), we recall the following
lemma in \cite{MixedHybrid}.

\begin{lemma}\cite[Sections 8.4.2 and 8.7.1]{MixedHybrid}\label{Lemma_Mini_Interpolation_Operator}
For the Mini mixed finite element space (\ref{Mini_Mixed_FEM}), there exists an interpolation operator $\pi_h: \mathbf V\rightarrow \mathbf V_h$ such that
\begin{eqnarray}
&&b(\mathbf v-\pi_h\mathbf v, q_h)=0,\ \ \ \forall \mathbf v\in\mathbf V,\ \forall q_h\in Q_h,\\
&&\|\pi_h \mathbf{v}\|_{\mathbf V} \leq C_{\pi}\ \|{\mathbf{v}}\|_{\mathbf V},
\quad \forall \mathbf{v} \in \mathbf V.
\end{eqnarray}
\end{lemma}

\begin{corollary}\label{corollary2.1}
For the Mini mixed finite element discretization, the following locking-free a priori error estimate holds
\begin{eqnarray*}
\|\mathbf{u} - \mathbf{u}_{h}\|_{\mathbf{V}} + \| p - p_{h}\|_{Q}
\leq C_{\rm err} \left( \inf_{\mathbf{v}_{h} \in \mathbf{V}_{h}}
\| \mathbf{u} - \mathbf{v}_h\|_{\mathbf{V}} + \inf_{q_{h} \in Q_{h}} \| p - q_h\|_{Q} \right),
\end{eqnarray*}
where $C_{\rm err}$ is independent of $\underline{\lambda}$ and the mesh size $h$. If $\mathbf u\in (H^2(\Omega))^d$ and $p\in H^1(\Omega)$, the Mini mixed finite element solution
$(\mathbf u_h, p_h)$ has the following optimal convergence order estimate
\begin{eqnarray}
\|\mathbf{u} - \mathbf{u}_{h}\|_{\mathbf{V}} + \| p - p_{h}\|_{Q} \leq Ch(\|\mathbf u\|_2+\|p\|_1).
\end{eqnarray}
\end{corollary}
\section{Mini mixed finite element discretization for linear elasticity eigenvalue problems}\label{sec:EigenAnalysis}

\subsection{Problem formulation}
Consider the following linear elasticity eigenvalue problem: Find $(\lambda,\mathbf{u})\in\mathbb{R}\times\mathbf{V}$ such that $\|\mathbf{u}\|_{0} = 1$ and
\begin{equation}\label{eq:EigenProblem}
\begin{cases}
-\underline{\mu}\Delta \mathbf{u} - (\underline{\lambda} + \underline{\mu}) \nabla (\mathrm{div} \mathbf{u}) = \lambda \mathbf{u}, & \text{in } \Omega, \\
\mathbf{u} = \mathbf{0}, & \text{on } \partial \Omega.
\end{cases}
\end{equation}
The corresponding mixed variational formulation is: Find $(\lambda,\mathbf{u},p)\in\mathbb{R}\times\mathbf{V}\times Q$ such that $\|\mathbf{u}\|_{0} = 1$ and
\begin{equation}\label{eq:MixedEigenForm}
\begin{cases}
a(\mathbf{u},\mathbf{v}) + b(\mathbf{v},p) = \lambda m(\mathbf{u},\mathbf{v}), & \forall \mathbf{v}\in\mathbf{V}, \\
b(\mathbf{u},q) - c(p,q) = 0, & \forall q\in Q,
\end{cases}
\end{equation}
where $m(\mathbf{u},\mathbf{v}) := \int_\Omega \mathbf{u}\cdot\mathbf{v}\,{\rm d}x$ is the mass bilinear form.

Before proceeding with the analysis, we briefly examine the relationship between the distinct eigenpairs
$(\lambda_{i}, \mathbf{u}_i, p_i)$ and $(\lambda_{j}, \mathbf{u}_j, p_j)$. From \eqref{eq:MixedEigenForm}, the following four equations hold simultaneously
\begin{align}
a(\mathbf{u}_i,\mathbf{u}_j) + b(\mathbf{u}_j, p_i) &= \lambda_{i} m(\mathbf{u}_i,\mathbf{u}_j), \label{eq:i1} \\
b(\mathbf{u}_i, p_j) - c(p_i, p_j) &= 0, \label{eq:i2}\\
a(\mathbf{u}_j,\mathbf{u}_i) + b(\mathbf{u}_i, p_j) &= \lambda_{j} m(\mathbf{u}_j,\mathbf{u}_i), \label{eq:j1}\\
b(\mathbf{u}_j, p_i) - c(p_j, p_i) &= 0. \label{eq:j2}
\end{align}
Subtracting \eqref{eq:j1} from \eqref{eq:i1}, using \eqref{eq:i2} and \eqref{eq:j2}, we derive the following orthogonality equality
\begin{align}
(\lambda_i - \lambda_j) m(\mathbf{u}_i, \mathbf{u}_j)
= b(\mathbf u_j, p_i)-b(\mathbf u_i, p_j) =
c(p_j,p_i)-c(p_i,p_j) = 0.
\end{align}
This implies that $m(\mathbf{u}_{i}, \mathbf{u}_j) = 0$ when $\lambda_i \neq \lambda_j$.
By orthogonalizing the eigenfunctions
within the eigenspaces corresponding to the multiple eigenvalues with respect to the inner product $m(\cdot,\cdot)$, we construct an orthonormal basis of eigenfunctions
under the inner product $m(\cdot,\cdot)$. Specifically, the eigenvalue problem (\ref{eq:MixedEigenForm})
has the eigenvalue sequence (cf. \cite{BabuskaOsborn_1989,Chatelin})
$$0<\lambda_{1} \leq \lambda_{2} \leq \cdots \leq \lambda_{k}\leq \cdots,\
\lim_{k\rightarrow \infty} \lambda_k = \infty,$$
and the corresponding eigenfunction sequence
$$(\mathbf{u}_{1}, p_1), (\mathbf{u}_{2}, p_2), \cdots, (\mathbf{u}_{k}, p_k), \cdots,$$
satisfying $m(\mathbf{u}_{i}, \mathbf{u}_{j}) = \delta_{i,j}$, with $\delta_{i,j}$
denoting the Kronecker function.

Let $\Phi = (\mathbf{u}, p) \in \mathcal{X} := \mathbf{V} \times Q$ and
$\Psi = (\mathbf{v}, q) \in \mathcal{X}$. Define the symmetric bilinear form
$$A(\Phi, \Psi) := a(\mathbf{u}, \mathbf{v}) + b(\mathbf{v}, p) + b(\mathbf{u}, q) - c(p, q).$$
Then, the eigenvalue problem \eqref{eq:MixedEigenForm} can be equivalently reformulated as:
Find $(\lambda, \Phi) \in \mathbb{R} \times \mathcal{X}$
such that $m(\mathbf{u}, \mathbf{u}) = 1$ and
\begin{equation}\label{eq:EigenValueProblem}
A(\Phi, \Psi) = \lambda m(\mathbf{u}, \mathbf{v}), \quad \forall\, \Psi = (\mathbf{v}, q) \in \mathcal{X}.
\end{equation}

Similarly, based on the preceding regularity results and variational theory, the eigenvalue problem \eqref{eq:EigenValueProblem}
admits an increasing sequence of eigenvalues $\{\lambda_j\}$ (cf. \cite{BabuskaOsborn_1989,Chatelin})
\begin{equation}\label{eq:eigenvalue_relationship}
0<\lambda_{1} \leq \lambda_{2} \leq \cdots \leq \lambda_{k}\leq \cdots, \
\lim_{k\rightarrow \infty} \lambda_k = \infty,
\end{equation}
along with an associated orthonormal eigenfunction system
$$\Phi_{1}, \Phi_{2}. \cdots, \Phi_{k}, \cdots,$$
where $\Phi_k = (\mathbf{u}_k, p_k)$ satisfies $m(\mathbf{u}_i, \mathbf{u}_j) = \delta_{i,j}$.
In the eigenvalue sequence $\{\lambda_j\}$,
each eigenvalue $\lambda_j$ is repeated according to its geometric multiplicity.

For the subsequent analysis, we state the following equality with respect to the smallest eigenvalue $\lambda_1$ (cf. \cite{BabuskaOsborn_Book})
\begin{equation}\label{eq:lambda1}
\lambda_{1} = \min_{0\not= \Psi \in \mathcal{X}} \frac{A(\Psi, \Psi)}{m(\mathbf{v}, \mathbf{v})}.
\end{equation}

Let $\mathcal{X}_h := \mathbf{V}_h \times Q_h \subset \mathcal{X}$ be the Mini mixed finite element space.
Discretizing the problem \eqref{eq:MixedEigenForm} with the Mini mixed finite element method yields:
Find \((\bar{\lambda}_h, \bar{\mathbf{u}}_h, \bar{p}_h) \in \mathbb{R} \times \mathbf{V}_h \times Q_h\)
such that \(m(\bar{\mathbf{u}}_h, \bar{\mathbf{u}}_h) = 1\) and
\begin{equation}\label{eq:MixedEigenFormDiscrate}
\begin{cases}
a(\bar{\mathbf{u}}_h,\mathbf{v}_h) + b(\mathbf{v}_h,\bar{p}_h) = \bar{\lambda}_h m(\bar{\mathbf{u}}_h,\mathbf{v}_h), & \forall \mathbf{v}_h\in\mathbf{V}_h, \\
b(\bar{\mathbf{u}}_h,q_h) - c(\bar{p}_h,q_h) = 0, & \forall q_h\in Q_h.
\end{cases}
\end{equation}

Let $\Phi_h = (\mathbf{u}_h, p_h)$ and $\Psi_h = (\mathbf{v}_h, q_h) \in \mathcal{X}_h$.
The discrete problem \eqref{eq:MixedEigenFormDiscrate} can be equivalently written as:
Find $(\bar{\lambda}_h, \bar{\Phi}_h) \in \mathbb{R} \times \mathcal{X}_h$ such that $m(\bar{\mathbf{u}}_h, \bar{\mathbf{u}}_h) = 1$ and
\begin{equation}\label{eq:EigenValueProblem1}
A(\bar{\Phi}_h, \Psi_h) = \bar{\lambda}_h m(\bar{\mathbf{u}}_h, \mathbf{v}_h), \quad \forall\, \Psi_h \in \mathcal{X}_h.
\end{equation}

By the max-min principle \cite{BabuskaOsborn_1989,BabuskaOsborn_Book}, we obtain the following upper bound property for eigenvalues
\begin{equation}\label{eq:maxmin}
\bar\lambda_{i,h} \geq \lambda_{i},\quad i=1,2,\cdots, N_h,
\end{equation}
where $N_h$ is the dimension of $\mathcal{X}_h$.

For the convenience of subsequent error analysis in the next section, let $(\bar\lambda_{i,h}, \bar\Phi_{i,h})$ denote the numerical solution solved from \eqref{eq:EigenValueProblem1}.
Define the Aubin-Nitsche estimator as
\begin{equation}\label{eq:AubinNitsche}
\eta_{\mathcal{X}}(\mathcal{X}_h) =
\sup_{\substack{\mathbf{F}\in \mathbf{W} \\ \|\mathbf{F}\|_{0} = 1}} \inf_{\Psi_h\in \mathcal{X}_h} \|T\mathbf{F}-\Psi_h\|_{\mathcal{X}},
\end{equation}
where $\mathbf{W} = (L^2(\Omega))^2$, and the operator
$T: \mathbf{W}\rightarrow \mathcal{X}$ is defined by
\begin{equation}\label{eq:linearpro}
A(T\mathbf{F},\Psi) = m(\mathbf{F},\mathbf{v}),\quad \forall \Psi=(\mathbf{v},q) \in \mathcal{X},\
\ {\rm for}\ \mathbf{F}\in \mathbf{W}.
\end{equation}

In order to present the corresponding error estimate,
we define the finite element projection operator $\mathcal{P}_h: \mathcal{X} \rightarrow \mathcal{X}_h$
such that $\mathcal{P}_h \Phi= \left( \mathcal{P}^{u}_h \Phi, \mathcal{P}^{p}_h \Phi \right)$ satisfying
\begin{equation}\label{eq:Energy_Projection}
\begin{cases}
a(\mathcal{P}^{u}_h \Phi,\mathbf{v}_h) + b(\mathbf{v}_h, \mathcal{P}^{p}_h \Phi)
= \lambda m(\mathbf{u}, \mathbf{v}_h), & \forall \mathbf{v}_h\in\mathbf{V}_h, \\
b(\mathcal{P}^{u}_h \Phi, q_h) - c(\mathcal{P}^{p}_h \Phi,q_h) = 0, & \forall q_h\in Q_h.
\end{cases}
\end{equation}

The following error estimates for the finite element projection hold.
\begin{lemma}\label{lem:AubinNitsche}
For any $\Phi \in \mathcal{X}$, the finite element projection operator $\mathcal{P}_h$ satisfies
\begin{eqnarray}
\|\mathbf{u} - \mathcal{P}^{u}_h \Phi\|_{\mathbf{V}}
+ \|p - \mathcal{P}^{p}_h \Phi\|_{Q} &\leq& C_{\rm err}
\left( \inf_{\mathbf{v}_{h}\in \mathbf{V}_{h}} \|\mathbf{u} - \mathbf{v}_{h}\|_{\mathbf{V}}
+ \inf_{q_{h}\in Q_{h}} \|p - q_{h}\|_{Q} \right), \label{Projection_Error_Estimate} \\
\|\mathbf{u}-\mathcal{P}^{u}_h \Phi \|_{0} &\leq& C \eta_{\mathcal{X}}(\mathcal{X}_h) \|\Phi-\mathcal{P}_h\Phi\|_{\mathcal{X}}, \label{Projection_Aubin_Nitsche}
\end{eqnarray}
where the constant $C$ is independent of $\underline{\lambda}$ and the mesh size $h$.
\end{lemma}
\begin{proof}
From Corollary \ref{corollary2.1} and the definition of finite element projection operator in \eqref{eq:Energy_Projection}, we can directly obtain \eqref{Projection_Error_Estimate}. 

By the definition of finite element projection operator in \eqref{eq:Energy_Projection},
for any $\Psi_h\in \mathbf{V}_h \times Q_h$, we have
\begin{align*}
&\|\mathbf{u}-\mathcal{P}^{u}_h \Phi \|_{0} = \sup_{\substack{\mathbf{F}\in \mathbf{W} \\
\|\mathbf{F}\|_{0} = 1}} |m(\mathbf{u}-\mathcal{P}^{u}_h \Phi, \mathbf{F})| \\
&= \sup_{\substack{\mathbf{F}\in \mathbf{W} \\
\|\mathbf{F}\|_{0} = 1}} |A(\Phi-\mathcal{P}_h\Phi,T\mathbf{F})|
= \sup_{\substack{\mathbf{F}\in \mathbf{W} \\
\|\mathbf{F}\|_{0} = 1}} |A(\Phi-\mathcal{P}_h\Phi,T\mathbf{F}-\Psi_h)| \\
&\leq C\sup_{\substack{\mathbf{F}\in \mathbf{W} \\
\|\mathbf{F}\|_{0} = 1}} \|\Phi-\mathcal{P}_h\Phi\|_{\mathcal{X}}\|T\mathbf{F}-\Psi_h\|_{\mathcal{X}} \\
&= C \|\Phi-\mathcal{P}_h\Phi\|_{\mathcal{X}} \sup_{\substack{\mathbf{F}\in \mathbf{W} \\
\|\mathbf{F}\|_{0} = 1}} \|T\mathbf{F}-\Psi_h\|_{\mathcal{X}},
\end{align*}
where $C$ is independent of $\underline{\lambda}$ and the mesh size $h$.
Using the arbitrariness of $\Psi_h$ and \eqref{eq:AubinNitsche}, we obtain
\begin{align*}
\|\mathbf{u}-\mathcal{P}^{u}_h \Phi \|_{0} &\leq C \|\Phi-\mathcal{P}_h\Phi\|_{\mathcal{X}}  \sup_{\substack{\mathbf{F}\in \mathbf{W} \\
\|\mathbf{F}\|_{0} = 1}} \inf_{\Psi_h\in \mathcal{X}_h} \|T\mathbf{F}-\Psi_h\|_{\mathcal{X}}
= C \eta_{\mathcal{X}}(\mathcal{X}_h) \|\Phi-\mathcal{P}_h\Phi\|_{\mathcal{X}},
\end{align*}
which is the desired result (\ref{Projection_Aubin_Nitsche}) and the proof is completed.
\end{proof}

To analyze the subspace projection errors, we introduce the following lemma with its proof.
\begin{lemma}\label{lemm:tmpB}
For the exact eigenpair $(\lambda, \Phi)$ of \eqref{eq:EigenValueProblem} and discrete eigenpair $(\bar\lambda_{j,h}, \bar\Phi_{j,h})$ of \eqref{eq:EigenValueProblem1},
the following identity holds
\begin{equation}
(\bar\lambda_{j,h}-\lambda)m(\mathcal{P}^{u}_h \Phi, \bar{\mathbf{u}}_{j,h}) =  \lambda m(\mathbf{u}- \mathcal{P}^{u}_h \Phi, \bar{\mathbf{u}}_{j,h}).
\end{equation}
\end{lemma}
\begin{proof}
It suffices to prove
$$\bar\lambda_{j,h}m(\mathcal{P}^{u}_h \Phi, \bar{\mathbf{u}}_{j,h}) = \lambda m(\mathbf{u}, \bar{\mathbf{u}}_{j,h}),$$
since $-\lambda m(\mathcal{P}^{u}_h \Phi, \bar{\mathbf{u}}_{j,h})$ is canceled on both sides. Using	(\ref{eq:EigenValueProblem1}) and (\ref{eq:Energy_Projection}), we derive
$$\bar\lambda_{j,h}m(\mathcal{P}^{u}_h \Phi, \bar{\mathbf{u}}_{j,h})
= A(\mathcal{P}_h \Phi, \bar\Phi_{j,h}) = A(\Phi, \bar\Phi_{j,h})
= \lambda m(\mathbf{u}, \bar{\mathbf{u}}_{j,h}),$$
which completes the proof.
\end{proof}

\subsection{Error estimates for eigenfunction}\label{subsec:SingleEigenEstimate}
Let $(\lambda, \Phi)$ be the exact eigenpair of the eigenvalue problem (\ref{eq:EigenValueProblem}),
and assume that the eigenpair approximation
$(\bar\lambda_{i,h}, \bar\Phi_{i,h})$ satisfies $\bar\mu_{i,h} = 1/\bar\lambda_{i,h}$
being the closest to $\mu = 1/\lambda$.
Now, we are concerned with the error estimates for eigenfunction.

To facilitate the subsequent analysis, we define the following spectral projection operators.
Based on the definition of finite element projection \eqref{eq:Energy_Projection},
let $\mathcal{P}_h \Psi = (\mathcal{P}_h^u \Psi, \mathcal{P}_h^p \Psi)$
be the solution of \eqref{eq:Energy_Projection}.
Since $\mathcal{P}_h^u \Psi\in {\rm span}\{\bar{\mathbf{u}}_{1,h}, \cdots,
\bar{\mathbf{u}}_{N_h,h}\}$, we have the orthogonal expansion
\begin{equation}\label{OrthogonalExpansion1}
\mathcal{P}_h^u \Psi = \sum_{j=1}^{N_h} \alpha_j \bar{\mathbf{u}}_{j,h},
\end{equation}
where $\alpha_j = m(\mathcal{P}_h^u \Psi, \bar{\mathbf{u}}_{j,h}), j=1,2,\cdots, N_h$.
For the same coefficients $\alpha_j$, the following equality holds
\begin{equation}\label{OrthogonalExpansion2}
\mathcal{P}_h^p \Psi = \sum_{j=1}^{N_h} \alpha_j \bar{p}_{j,h}.
\end{equation}
A brief explanation for the establishment of \eqref{OrthogonalExpansion2} is given below.

From \eqref{eq:MixedEigenFormDiscrate}, for any $q_h\in Q_h$, we have
\begin{equation}\label{eq:proof4Express1}
b\left(\sum_{j=1}^{N_h} \alpha_j \bar{\mathbf{u}}_{j,h}, q_h\right)
- c\left(\sum_{j=1}^{N_h} \alpha_j \bar{p}_{j,h},q_h\right)
= \sum_{j=1}^{N_h} \alpha_j \left( b(\bar{\mathbf{u}}_{j,h}, q_h)
- c(\bar{p}_{j,h}, q_h) \right) = 0.
\end{equation}
Furthermore, for any $\bar{\mathbf{u}}_{k,h}\in \mathbf{V}_h$, $k=1,\cdots,N_h$,
using \eqref{eq:MixedEigenFormDiscrate}, \eqref{eq:Energy_Projection}
and \eqref{OrthogonalExpansion1}, we derive
\begin{align}\label{eq:proof4Express2}
& a\left(\sum_{j=1}^{N_h} \alpha_j \bar{\mathbf{u}}_{j,h},\bar{\mathbf{u}}_{k,h} \right) + b\left(\bar{\mathbf{u}}_{k,h}, \sum_{j=1}^{N_h} \alpha_j \bar{p}_{j,h}\right)  \nonumber \\
&= \sum_{j=1}^{N_h} \alpha_j \left( a(\bar{\mathbf{u}}_{j,h}, \bar{\mathbf{u}}_{k,h})
+ b(\bar{\mathbf{u}}_{k,h}, \bar{p}_{j,h}) \right) \nonumber \\
&= \sum_{j=1}^{N_h} \alpha_j \bar{\lambda}_{j,h} m(\bar{\mathbf{u}}_{j,h},\bar{\mathbf{u}}_{k,h})
= \alpha_k \bar{\lambda}_{k,h}
= \bar{\lambda}_{k,h} m(\bar{\mathbf{u}}_{k,h}, \mathcal{P}_h^u \Psi) \nonumber \\
&= a(\bar{\mathbf{u}}_{k,h}, \mathcal{P}_h^u \Psi) + b(\mathcal{P}_h^u \Psi, \bar{p}_{k,h})
= a(\mathcal{P}_h^u \Psi, \bar{\mathbf{u}}_{k,h}) + c(\mathcal{P}_h^p \Psi, \bar{p}_{k,h}) \nonumber \\
&= a(\mathcal{P}_h^u \Psi, \bar{\mathbf{u}}_{k,h}) + c(\bar{p}_{k,h}, \mathcal{P}_h^p \Psi)
= a(\mathcal{P}_h^u \Psi, \bar{\mathbf{u}}_{k,h}) + b(\bar{\mathbf{u}}_{k,h}, \mathcal{P}_h^p \Psi).
\end{align}
From \eqref{eq:proof4Express1} and \eqref{eq:proof4Express2},
the pair $\left(\sum_{j=1}^{N_h} \alpha_j \bar{\mathbf{u}}_{j,h},
\sum_{j=1}^{N_h} \alpha_j \bar{p}_{j,h}\right)$ constitutes an approximating solution of \eqref{eq:Energy_Projection}.
Since $\mathcal{P}_h \Psi = (\mathcal{P}_h^u \Psi, \mathcal{P}_h^p \Psi)$ is also an approximating solution of \eqref{eq:Energy_Projection}, the uniqueness of solution for (\ref{eq:Energy_Projection}) implies
\begin{equation}\label{eq:tyfjtmp}
\mathcal{P}_h \Psi = (\mathcal{P}_h^u \Psi, \mathcal{P}_h^p \Psi)
= \sum_{j=1}^{N_h} \alpha_j  \begin{pmatrix}
\bar{\mathbf{u}}_{j,h} , \bar{p}_{j,h}
\end{pmatrix},
\end{equation}
where $\alpha_j = m(\mathcal{P}_h^u \Psi, \bar{\mathbf{u}}_{j,h})$, for $j=1,2,\cdots, N_h$.
Thus, the expansion \eqref{OrthogonalExpansion2} for $\mathcal P_h^p\Psi$ holds.

Next, we define a spectral projection operator $\Pi_{i,h}^u: {\rm span}\{\bar{\mathbf{u}}_{1,h},\cdots,\bar{\mathbf{u}}_{N_h,h}\} \rightarrow {\rm span}\{\bar{\mathbf{u}}_{i,h}\}$ as the $m(\cdot,\cdot)$-orthogonal projection
\begin{equation}\label{project_m}
m((I - \Pi_{i,h}^u) \mathbf{v}_h, \bar{\mathbf{u}}_{i,h}) = 0,\quad \forall \mathbf{v}_h \in {\rm span}\{\bar{\mathbf{u}}_{1,h},\cdots,\bar{\mathbf{u}}_{N_h,h}\},
\end{equation}
where $I$ represents the unit operator.

Define the spectral projection operator $E_{i,h} = \Pi_{i,h}\mathcal{P}_h$ satisfying
\begin{equation}\label{project_eig_1}
E_{i,h} \Psi = \Pi_{i,h} \mathcal{P}_h \Psi = \alpha_i \bar{\Phi}_{i,h},
\end{equation}
where $\alpha_i = m(\mathcal{P}_h^u \Psi, \bar{\mathbf{u}}_{i,h})$ and $\Pi_{i,h} \Phi_{h} = (\Pi_{i,h}^u \mathbf{v}_{h}, \Pi_{i,h}^p \Phi_{h})$.

\begin{lemma}\label{lem:project_eig_1}
For any $\Psi_h = (\mathbf{v}_h, q_h)\in {\rm span}\{\bar{\Phi}_{1,h},\cdots,\bar{\Phi}_{N_h,h}\}$,
the projection operator $\Pi_{i,h}: {\rm span}\{\bar{\Phi}_{1,h},\cdots,\bar{\Phi}_{N_h,h}\} \rightarrow {\rm span}\{\bar{\Phi}_{i,h}\}$ is a bounded linear operator and satisfies
\begin{eqnarray}
&&A(\Pi_{i,h} \Psi_h ,\bar{\Phi}_{i,h}) = A(\Psi_h,\bar{\Phi}_{i,h}),\label{Projection_Orognality}\\
&&\|\Pi_{i,h}\Psi_h\|_{\mathcal{X}} \leq C \| \Psi_h\|_{\mathcal{X}},\label{Projection_Boundness}
\end{eqnarray}
where $C$ is independent of the Lam\'{e} constant $\underline{\lambda}$.
\end{lemma}
\begin{proof}
From the equation \eqref{project_m}, we have
$$m((I - \Pi_{i,h}^u) \mathbf{v}_h, \bar{\mathbf{u}}_{i,h}) = 0.$$
Combined with \eqref{eq:EigenValueProblem1}, it follows that
\begin{align*}
0 &= \bar\lambda_{i,h}m((I - \Pi_{i,h}^u) \mathbf{v}_h, \bar{\mathbf{u}}_{i,h})
= \bar{\lambda}_{i,h} m( \bar{\mathbf{u}}_{i,h}, (I - \Pi_{i,h}^u) \mathbf{v}_h) \\
&= A(\bar{\Phi}_{i,h}, (I - \Pi_{i,h}) \Psi_h)
= A((I - \Pi_{i,h}) \Psi_h, \bar{\Phi}_{i,h}) \\
&= A(\Psi_h, \bar{\Phi}_{i,h}) - A(\Pi_{i,h} \Psi_h, \bar{\Phi}_{i,h}).
\end{align*}
Thus, the equality (\ref{Projection_Orognality}) holds.

By \eqref{eq:tyfjtmp}, let $\Psi_h = \sum\limits_{j=1}^{N_h} \alpha_j \bar{\Phi}_{j,h},$
where $\Psi_h = (\mathbf{v}_h, q_h)$ and $\alpha_j = m(\mathbf{v}_h, \bar{\mathbf{u}}_{j,h})$ for $j=1,2,\cdots, N_h$. Then,
\begin{align*}
&\|\Pi_{i,h} \Psi_h\|_{\mathcal{X}}^2
= \left\|  \alpha_i \bar{\Phi}_{i,h} \right\|_{\mathcal{X}}^2
= \alpha_{i}^2 \left\| \bar{\Phi}_{i,h} \right\|_{\mathcal{X}}^2  \\
&= m(\mathbf{v}_h, \bar{\mathbf{u}}_{i,h})^2 \left\| \bar{\Phi}_{i,h} \right\|_{\mathcal{X}}^2
\leq \|\mathbf{v}_h\|_m^2 \left\| \bar{\Phi}_{i,h} \right\|_{\mathcal{X}}^2 \\
&\leq \| \Psi_h \|_{\mathcal{X}}^2 \left\| \bar{\Phi}_{i,h} \right\|_{\mathcal{X}}^2
\leq \bar\lambda_{i,h} \| \Psi_h \|_{\mathcal{X}}^2.
\end{align*}
Hence, the inequality (\ref{Projection_Boundness}) holds, which completes the proof.
\end{proof}

\begin{theorem}\label{thm:EigErr}
Let $(\lambda, \Phi)$ be the exact eigenpair of \eqref{eq:EigenValueProblem},
and assume that the approximating eigenpair $(\bar\lambda_{i,h}, \bar\Phi_{i,h})$
satisfies $\bar\mu_{i,h} = 1/\bar\lambda_{i,h}$ being the closest to $\mu = 1/\lambda$.
The spectral projection operator $E_{i,h}$ defined in \eqref{project_eig_1} yields the following error estimate
\begin{equation}\label{eq:dl1}
\|\Phi - E_{i,h} \Phi\|_{\mathcal{X}} \leq \left(1+C\frac{\sqrt{\mu_{1}} + 1}{\delta_{\lambda, h}} \eta_{\mathcal{X}}\left(\mathcal{X}_{h}\right)\right) \| \left(I-\mathcal{P}_{h}\right) \Phi\|_{\mathcal{X}},
\end{equation}
where $\eta_{\mathcal{X}}(\mathcal{X}_h)$ is defined in \eqref{eq:AubinNitsche},
and the spectral gap $\delta_{\lambda,h}$ is given by
$$\delta_{\lambda,h} := \min_{j\not=i} \left| \mu-\bar\mu_{j,h} \right| = \min_{j\not=i} \left| \frac{1}{\lambda}-\frac{1}{\bar\lambda_{j,h}}\right|.$$
Furthermore, the eigenfunction approximation $\bar{\mathbf{u}}_{i,h}$ satisfies the error estimate in the $L^2$-norm,
\begin{equation}\label{eq:dl2}
\|\mathbf{u}-E_{i, h}^u \Phi \|_{0} \leq C \left(1+\frac{\mu_{1}}{\delta_{\lambda, h}}\right) \eta_{\mathcal{X}}\left(\mathcal{X}_{h}\right)\| \left(I-\mathcal{P}_{h}\right) \Phi\|_{\mathcal{X}}.
\end{equation}
\end{theorem}
\begin{proof}
Since $(I-E_{i,h}^{u})\mathcal{P}_h^{u} \Phi\in V_h$ and $(I-E_{i,h}^{u})\mathcal{P}_h^{u}
\Phi \in {\rm span}\{\bar{\mathbf{u}}_{1,h}$, $\cdots$,$\bar{\mathbf{u}}_{i-1,h}$,
$\bar{\mathbf{u}}_{i+1,h}$,
$\cdots$, $\bar{\mathbf{u}}_{N_h,h}\}$, we obtain the orthogonal expansion
\begin{equation}\label{eq:EPh1}
(I-E_{i,h}^{u})\mathcal{P}_h^{u} \Phi = \sum_{j\not=i}^{N_{h}} \alpha_{j} \bar{\mathbf{u}}_{j, h},
\end{equation}
where $\alpha_j = m(\mathcal{P}_{h}^{u} \Phi,\bar{\mathbf{u}}_{j, h})$.
By Lemma \ref{lemm:tmpB}, we have
\begin{align}
\alpha_{j} = m\left(\mathcal{P}_{h}^{u} \Phi,\bar{\mathbf{u}}_{j, h}\right)
=\frac{\lambda}{\bar{\lambda}_{j, h}-\lambda}
m\left(\mathbf{u}-\mathcal{P}_{h}^{u} \Phi, \bar{\mathbf{u}}_{j, h}\right)
=\frac{\bar{\mu}_{j, h}}{\mu - \bar{\mu}_{j, h}}
m\left(\mathbf{u}-\mathcal{P}_{h}^{u} \Phi, \bar{\mathbf{u}}_{j, h}\right).\label{eq:alpha1j}
\end{align}

Combining \eqref{eq:MixedEigenFormDiscrate}, \eqref{eq:Energy_Projection},
\eqref{OrthogonalExpansion2}, the orthogonal expansion \eqref{eq:EPh1}
and \eqref{eq:alpha1j} leads to the following inequalities
\begin{align}
&\left\|\left(I-E_{i, h}^p\right) \mathcal{P}_{h}^p \Phi\right\|_Q = \left\| \sum_{j\not=i}^{N_{h}} \alpha_{j} \bar{p}_{j, h} \right\|_Q
\leq C \sum_{j\not=i}^{N_{h}} \alpha_{j} \bar{\lambda}_{j,h} \nonumber \\
&= C \sum_{j\not=i}^{N_{h}} \frac{1}{\mu - \bar{\mu}_{j, h}} m\left(\mathbf{u}-\mathcal{P}_{h}^{u} \Phi, \bar{\mathbf{u}}_{j, h}\right) \nonumber \\
&\leq \frac{C}{\delta_{\lambda, h}} \sum_{j\not=i}^{N_{h}} m\left(\mathbf{u}-\mathcal{P}_{h}^{u} \Phi, \bar{\mathbf{u}}_{j, h}\right) \leq \frac{C}{\delta_{\lambda, h}}\|\mathbf{u}-\mathcal{P}_{h}^{u} \Phi\|_{0}.\label{eq:IEPp}
\end{align}

Similarly, combining \eqref{eq:MixedEigenFormDiscrate}, \eqref{eq:maxmin},  \eqref{eq:Energy_Projection}, \eqref{eq:EPh1} and \eqref{eq:alpha1j}, we obtain
\begin{align}
&a\left(\left(I-E_{i, h}^u\right) \mathcal{P}_{h}^u \Phi, \left(I-E_{i, h}^u\right) \mathcal{P}_{h}^u \Phi\right) = a\left(\sum_{j\not=i}^{N_{h}} \alpha_{j} \bar{\mathbf{u}}_{j, h}, \sum_{j\not=i}^{N_{h}} \alpha_{j} \bar{\mathbf{u}}_{j, h}\right) \nonumber \\
&\leq a\left(\sum_{j\not=i}^{N_{h}} \alpha_{j} \bar{\mathbf{u}}_{j, h}, \sum_{j\not=i}^{N_{h}} \alpha_{j} \bar{\mathbf{u}}_{j, h}\right) + b\left(\sum_{j\not=i}^{N_{h}} \alpha_{j} \bar{\mathbf{u}}_{j, h}, \sum_{j\not=i}^{N_{h}} \alpha_{j} \bar{p}_{j, h}\right) \nonumber \\
&\quad  - b\left(\sum_{j\not=i}^{N_{h}} \alpha_{j} \bar{\mathbf{u}}_{j, h}, \sum_{j\not=i}^{N_{h}} \alpha_{j} \bar{p}_{j, h}\right) + c\left(\sum_{j\not=i}^{N_{h}} \alpha_{j} \bar{p}_{j, h}, \sum_{j\not=i}^{N_{h}} \alpha_{j} \bar{p}_{j, h}\right) \nonumber \\
&= \sum_{j\not=i}^{N_{h}} \sum_{t\not=i}^{N_{h}} \alpha_{j}\alpha_{t} \left( a(\bar{\mathbf{u}}_{j, h}, \bar{\mathbf{u}}_{t, h}) + b(\bar{\mathbf{u}}_{t, h}, \bar{p}_{j, h}) - b(\bar{\mathbf{u}}_{j, h}, \bar{p}_{t, h}) + c(\bar{p}_{j, h}, \bar{p}_{t, h}) \right) \nonumber \\
&= \sum_{j\not=i}^{N_{h}} \sum_{t\not=i}^{N_{h}} \alpha_{j}\alpha_{t} \bar{\lambda}_{j,h} m(\bar{\mathbf{u}}_{j, h}, \bar{\mathbf{u}}_{t, h}) = \sum_{j\not=i}^{N_{h}} \alpha_{j}^2 \bar{\lambda}_{j,h}  \nonumber \\
&=\sum_{j\not=i}^{N_{h}} \frac{\bar{\mu}_{j, h}}{(\mu - \bar{\mu}_{j, h})^2} m\left(\mathbf{u}-\mathcal{P}_{h}^{u} \Phi, \bar{\mathbf{u}}_{j, h}\right)^2 \nonumber\\
&\leq \frac{\bar{\mu}_1}{\delta_{\lambda, h}^{2}} \sum_{j\not=i}^{N_{h}} m\left(\mathbf{u}-\mathcal{P}_{h}^{u} \Phi, \bar{\mathbf{u}}_{j, h}\right)^2 \leq \frac{\mu_{1}}{\delta_{\lambda, h}^{2}}\|\mathbf{u}-\mathcal{P}_{h}^{u} \Phi\|_{0}^{2}.\label{eq:IEP1}
\end{align}
By Lemma \ref{lem:AubinNitsche}, \eqref{eq:IEPp} and \eqref{eq:IEP1}, we derive
\begin{align}
&\|\left(I-E_{i, h}\right) \mathcal{P}_{h} \Phi\|_{\mathcal{X}} \leq \left\|\left(I-E_{i, h}^u\right) \mathcal{P}_{h}^u \Phi\right\|_{\mathbf{V}} + \left\|\left(I-E_{i, h}^p\right) \mathcal{P}_{h}^p \Phi \right\|_Q \nonumber \\
&\leq C \frac{\sqrt{\mu_1}}{\delta_{\lambda, h}}\|\mathbf{u}-\mathcal{P}_{h}^{u} \Phi\|_{0} + \frac{C}{\delta_{\lambda, h}}\|\mathbf{u}-\mathcal{P}_{h}^{u} \Phi\|_{0} \nonumber \\
&\leq C \frac{\sqrt{\mu_{1}} + 1}{\delta_{\lambda, h}} \eta_{\mathcal{X}}(\mathcal{X}_h) \|\Phi-\mathcal{P}_h\Phi\|_{\mathcal{X}}.\label{eq:2311}
\end{align}
Using \eqref{eq:2311} and the triangle inequality, we have
\begin{align*}
&\|\Phi - E_{i,h} \Phi \|_{\mathcal{X}} \leq \|\Phi - \mathcal{P}_{h} \Phi\|_{\mathcal{X}} + \|\left(I-E_{i, h} \right) \mathcal{P}_{h} \Phi\|_{\mathcal{X}}\\
& \leq \left(1+ C \frac{\sqrt{\mu_1} + 1}{\delta_{\lambda, h}} \eta_{\mathcal{X}}\left(\mathcal{X}_{h}\right)\right) \|\left(I-\mathcal{P}_{h}\right) \Phi\|_{\mathcal{X}},
\end{align*}
which establishes the desired inequality \eqref{eq:dl1}.

Analogously, combining \eqref{eq:maxmin}, \eqref{eq:EPh1} and \eqref{eq:alpha1j}, we obtain
\begin{align*}
&\left\|\left(I-E_{i, h}^u\right) \mathcal{P}_{h}^u \Phi\right\|_{0}^{2}=\left\|\sum_{j\not=i}^{N_{h}} \alpha_{j} \bar{\mathbf{u}}_{j, h}\right\|_{0}^{2} = \sum_{j\not=i}^{N_{h}} \alpha_{j}^{2} \\
&=\sum_{j\not=i}^{N_{h}} \left(  \frac{\bar{\mu}_{j, h}}{\mu - \bar{\mu}_{j, h}} \right)^{2} m\left(\mathbf{u}-\mathcal{P}_{h}^{u} \Phi, \bar{\mathbf{u}}_{j, h}\right)^2\\
&\leq \frac{\bar\mu_{1,h}^2}{\delta_{\lambda, h}^{2}} \sum_{j\not=i}^{N_{h}}  m\left(\mathbf{u}-\mathcal{P}_{h}^{u} \Phi, \bar{\mathbf{u}}_{j, h}\right)^2 \\
&\leq \frac{\bar{\mu}_{1, h}^2}{\delta_{\lambda, h}^{2}}\|\mathbf{u}-\mathcal{P}_{h}^{u} \Phi\|_{0}^{2} \leq \frac{\mu_{1}^2}{\delta_{\lambda, h}^{2}}\|\mathbf{u}-\mathcal{P}_{h}^{u} \Phi\|_{0}^{2}.
\end{align*}
Thus, we derive the inequality
\begin{equation}\label{eq:bnormest1}
\left\|\left(I-E_{i, h}^u\right) \mathcal{P}_{h}^u \Phi\right\|_{0} \leq \dfrac{\mu_{1}}{\delta_{\lambda,h}} \|\mathbf{u}-\mathcal{P}_{h}^{u} \Phi\|_{0}.
\end{equation}
By Lemma \ref{lem:AubinNitsche}, \eqref{eq:bnormest1} and the triangle inequality, we conclude
\begin{align*}
&\|\mathbf{u}-E_{i, h}^u \Phi\|_{0} \leq\|\mathbf{u}-\mathcal{P}_{h}^u \Phi\|_{0}+\left\|\left(I-E_{i, h}^u\right) \mathcal{P}_{h}^u \Phi\right\|_{0} \\
&\leq\left(1+\frac{\mu_{1}}{\delta_{\lambda, h}}\right)\|\mathbf{u} - \mathcal{P}_{h}^u\Phi\|_{0} \leq C \left(1+\frac{\mu_{1}}{\delta_{\lambda, h}}\right) \eta_{\mathcal{X}}\left(\mathcal{X}_{h}\right)\|\Phi - \mathcal{P}_{h} \Phi\|_{\mathcal{X}},
\end{align*}
which is the desired result \eqref{eq:dl2} and the proof is completed.
\end{proof}

To simplify notation, assume that the eigenvalue gap $\delta_{\lambda,h}$ admits a uniform lower bound $\delta_{\lambda}$ independent of the mesh size $h$. This assumption is justified when the approximating eigenvalues achieve sufficient accuracy, i.e., when the mesh size $h$ is sufficiently small. Due to Theorem \ref{thm:EigErr}, we obtain the following error estimate.
\begin{corollary}\label{cor:Eigerr}
Under the conditions of Theorem \ref{thm:EigErr}, with $\delta_{\lambda,h}$ replaced by a mesh-independent lower bound $\delta_{\lambda}$, the following simplified error estimates hold
\begin{align*}
\|\Phi - E_{i,h} \Phi\|_{\mathcal{X}} &\leq \left(1 + C \frac{\sqrt{\mu_1}+1}{\delta_{\lambda}} \eta_{\mathcal{X}}\left(\mathcal{X}_{h}\right)\right) \| \left(I-\mathcal{P}_{h}\right) \Phi\|_{\mathcal{X}}, \\
\|\mathbf{u}-E_{i, h}^u \Phi\|_{0} &\leq C \left(1+\frac{\mu_{1}}{\delta_{\lambda}}\right) \eta_{\mathcal{X}}\left(\mathcal{X}_{h}\right)\| \left(I-\mathcal{P}_{h}\right) \Phi\|_{\mathcal{X}}.
\end{align*}
\end{corollary}

\begin{corollary}\label{cor:Eigerr1}
Under the conditions of Theorem \ref{thm:EigErr}, the following enhanced error estimates hold
\begin{align}
\left\|\lambda \mathbf{u}-\bar\lambda_{i,h} \bar{\mathbf{u}}_{i,h}\right\|_{0} &\leq C \left(1+\frac{\mu_{1}}{\delta_{\lambda}}\right) \eta_{\mathcal{X}}\left(\mathcal{X}_{h}\right)\| \left(I-\mathcal{P}_{h}\right) \Phi\|_{\mathcal{X}} + C \| \Phi -\bar{\Phi}_{i,h}\|_{\mathcal{X}}^2, \label{eq:corerr1}\\
\|\Phi - \bar\Phi_{i,h}\|_{\mathcal{X}} &\leq  C \| \left(I-\mathcal{P}_{h}\right) \Phi\|_{\mathcal{X}}. \label{eq:corerr2}
\end{align}
\end{corollary}
\begin{proof}
From \eqref{eq:MixedEigenForm} and \eqref{eq:Energy_Projection}, for any $\Psi_h \in \mathcal{X}_h$, we have
$$A(\mathcal{P}_h \Phi - \bar{\Phi}_{i,h}, \Psi_h) =A(\Phi- \bar{\Phi}_{i,h}, \Psi_h) = m(\lambda \mathbf{u} - \bar{\lambda}_{i,h}\bar{\mathbf{u}}_{i,h}, \mathbf{v}_h).$$
By Theorem \ref{thm:Wellposedness}, we obtain
\begin{equation}\label{coreq1}
\|\mathcal{P}_h \Phi - \bar{\Phi}_{i,h}\|_{\mathcal{X}} \leq C \|\lambda \mathbf{u} - \bar{\lambda}_{i,h}\bar{\mathbf{u}}_{i,h}\|_{0}.
\end{equation}
Using the orthogonality of eigenfunctions and Corollary \ref{cor:Eigerr}, the following estimates hold
\begin{align}\label{coreq2}
&\|\mathbf{u} - \bar{\mathbf{u}}_{i,h}\|_{0} \leq \|\mathbf{u} - E_{i,h}^u\Phi\|_{0} + \|E_{i,h}^u\Phi - \bar{\mathbf{u}}_{i,h}\|_{0} \nonumber \\
&= \|\mathbf{u} - E_{i,h}^u\Phi\|_{0} + \left| \|E_{i,h}^u\Phi\|_{0} - \|\bar{\mathbf{u}}_{i,h}\|_{0} \right| \nonumber \\
&=  \|\mathbf{u} - E_{i,h}^u\Phi\|_{0} + \left| \|E_{i,h}^u\Phi\|_{0} - 1 \right| \nonumber \\
&=  \|\mathbf{u} - E_{i,h}^u\Phi\|_{0} + \left| \|E_{i,h}^u\Phi\|_{0} - \|\mathbf{u}\|_{0} \right|
\leq 2 \|\mathbf{u} - E_{i,h}^u\Phi\|_{0} \nonumber \\
&\leq  C \left(1+\frac{\mu_{1}}{\delta_{\lambda}}\right) \eta_{\mathcal{X}}\left(\mathcal{X}_{h}\right)\|\left(I-\mathcal{P}_{h}\right) \Phi\|_{\mathcal{X}}.
\end{align}
Noting that for any $\Psi_h \in \mathcal{X}_h$, the equality $A(\Psi_h, \Psi_h) - \lambda m(\mathbf{v}_h, \mathbf{v}_h) = A(\Phi - \Psi_h, \Phi - \Psi_h) - \lambda m(\mathbf{u} - \mathbf{v}_h, \mathbf{u} - \mathbf{v}_h)$ holds. Substituting $\bar{\Phi}_{i,h}$ for $\Psi_h$ yields
\begin{align}\label{coreq3}
&|\lambda - \bar\lambda_{i,h}| \leq \frac{ A(\Phi -\bar{\Phi}_{i,h}, \Phi -\bar{\Phi}_{i,h})}{\|\bar{\mathbf{u}}_{i,h}\|_{0}^2} \leq C \|\Phi -\bar{\Phi}_{i,h}\|_{\mathcal{X}}^2.
\end{align}
Combining \eqref{coreq2} and \eqref{coreq3}, we present
\begin{align*}
& \left\|\lambda \mathbf{u}-\bar\lambda_{i,h} \bar{\mathbf{u}}_{i,h}\right\|_{0} \leq|\lambda|\left\| \mathbf{u} - \bar{\mathbf{u}}_{i,h} \right\|_{0} + \|\bar{\mathbf{u}}_{i,h}\|_{0} \mid \lambda-\bar\lambda_{i,h} | \nonumber \\
&\leq  C \left(1+\frac{\mu_{1}}{\delta_{\lambda}}\right) \eta_{\mathcal{X}}\left(\mathcal{X}_{h}\right)\|\left(I-\mathcal{P}_{h}\right) \Phi\|_{\mathcal{X}} + C \|\Phi -\bar{\Phi}_{i,h}\|_{\mathcal{X}}^2,
\end{align*}
which proves \eqref{eq:corerr1}.

Because of Lemma \ref{lem:project_eig_1}, Corollary \ref{cor:Eigerr}, \eqref{coreq1}, \eqref{coreq2}, \eqref{coreq3} and the triangle inequality, we obtain
\begin{align}\label{coreq4}
&\|\Phi - \bar\Phi_{i,h}\|_{\mathcal{X}} \leq \|\Phi - E_{i,h}\Phi\|_{\mathcal{X}} + \|E_{i,h}\Phi - \bar\Phi_{i,h}\|_{\mathcal{X}} \nonumber \\
&= \|\Phi - E_{i,h}\Phi\|_{\mathcal{X}} + \|\Pi_{i,h}\mathcal{P}_h\Phi - \bar\Phi_{i,h}\|_{\mathcal{X}} \nonumber\\
&\leq \|\Phi - E_{i,h}\Phi\|_{\mathcal{X}} + C \|\mathcal{P}_h\Phi - \bar\Phi_{i,h}\|_{\mathcal{X}} \nonumber\\
&\leq \|\Phi - E_{i,h}\Phi\|_{\mathcal{X}} + C \|\lambda \mathbf{u} - \bar{\lambda}_{i,h}\bar{\mathbf{u}}_{i,h}\|_{0}   \nonumber\\
&\leq  \|\Phi - E_{i,h}\Phi\|_{\mathcal{X}} + C(|\lambda|\left\| \mathbf{u} - \bar{\mathbf{u}}_{i,h} \right\|_{0} + \|\bar{\mathbf{u}}_{i,h}\|_{0} \mid \lambda-\bar\lambda_{i,h} |) \nonumber \\
&\leq  C \|\Phi - E_{i,h}\Phi\|_{\mathcal{X}} + C \left\| \mathbf{u} - \bar{\mathbf{u}}_{i,h} \right\|_{0} + C \|\Phi -\bar{\Phi}_{i,h}\|_{\mathcal{X}}^2 \nonumber \\
&\leq  C \|\left(I-\mathcal{P}_{h}\right) \Phi\|_{\mathcal{X}} + C \|\Phi -\bar{\Phi}_{i,h}\|_{\mathcal{X}}^2.
\end{align}
Observe that as $h\rightarrow 0$, $\|\Phi -\bar{\Phi}_{i,h}\|_{\mathcal{X}} \rightarrow 0$ and $\|\left(I-\mathcal{P}_{h}\right) \Phi\|_{\mathcal{X}} \rightarrow 0$. When $h$ is sufficiently small, \eqref{coreq4} implies
\begin{equation*}
\|\Phi - \bar\Phi_{i,h}\|_{\mathcal{X}} \leq  C \|\left(I-\mathcal{P}_{h}\right) \Phi\|_{\mathcal{X}},
\end{equation*}
which proves \eqref{eq:corerr2}, and the proof is finished.
\end{proof}

\section{Augmented subspace method based on Mini mixed finite element}\label{section4}

In this section, we design a non-nested augmented subspace method suitable for mixed finite element solution of linear elasticity eigenvalue problem. Unlike the augmented subspace or multi-level correction methods presented by Xu et al. \cite{XuHuangMa,MR4512632} for linear elasticity eigenvalue problem, our proposed algorithm here specifically addresses the locking phenomenon.

To define the non-nested augmented subspace method, we first construct a coarse mesh $\mathcal{T}_H$, a quasi-uniform triangulation of $\Omega$ with the mesh size $H$. Let $\mathcal{X}_H$ and $\mathcal{X}_h$ denote the Mini mixed finite element spaces defined on $\mathcal{T}_H$ and a refined mesh $\mathcal{T}_h$, respectively, where clearly $\mathcal{X}_H \not\subset \mathcal{X}_h$. The augmented subspace $\mathcal{X}_{H,h}$ is constructed by combining the functions from the coarse Mini mixed finite element space $\mathcal{X}_H$ and the fine space $\mathcal{X}_h$. Although $\mathcal{X}_H$ and $\mathcal{X}_h$ lack nested properties, the augmented subspace $\mathcal{X}_{H,h}$ forms a finite-dimensional subspace of $\mathcal{X}$. We subsequently demonstrate that the error estimates in Theorem \ref{thm:EigErr}, Corollary \ref{cor:Eigerr} and Corollary \ref{cor:Eigerr1} remain valid for $\mathcal{X}_{H,h}$. Based on this framework, we develop an augmented subspace algorithm and conduct the error analysis.

For a given eigenfunction approximation $\Phi_h^{(\ell)}$, we implement the augmented subspace iteration algorithm described by Algorithm \ref{Alg:Algorithm_1} to enhance the accuracy of $\Phi_h^{(\ell)}$. Here, the superscript $\ell$ denotes the iteration index. When $\ell = 1$, $(\lambda_h^{(\ell)}, \Phi_h^{(\ell)})$ represents the input eigenpair. Notably, $\lambda_h^{(\ell)}$ may approximate either a simple or multiple eigenvalue.

\begin{algorithm}[hbt!]
\caption{Augmented subspace iteration algorithm}\label{Alg:Algorithm_1}
\begin{enumerate}
\item \textbf{Initialization ($\ell=1$)}: Define $\widehat{\Phi}_h^{(\ell)} = \Phi_h^{(\ell)}$ and construct the augmented subspace $\mathcal{X}_{H,h}^{(\ell)} = \mathcal{X}_H + \text{span}\{\widehat{\Phi}_h^{(\ell)}\}$. Solve the eigenvalue problem: Find $(\lambda_h^{(\ell)}, \Phi_h^{(\ell)}) \in \mathbb{R} \times \mathcal{X}_{H,h}^{(\ell)}$ satisfying $\|\mathbf{u}_h^{(\ell)}\|_{0} = 1$ and
\begin{equation}\label{parallel_correct_eig_exact_1}
A(\Phi_h^{(\ell)}, \Psi_{H,h}) = \lambda_h^{(\ell)} m(\mathbf{u}_h^{(\ell)}, \mathbf{v}_{H,h}), \quad \forall \Psi_{H,h} \in \mathcal{X}_{H,h}^{(\ell)}.
\end{equation}

\item \textbf{Linear solution}: Find $\widehat{\Phi}_h^{(\ell+1)} \in \mathcal{X}_h$ satisfying
\begin{equation}\label{Linear_Equation_k_elastic}
A(\widehat{\Phi}_h^{(\ell+1)}, \Psi_h) = \lambda_h^{(\ell)} m(\mathbf{u}_h^{(\ell)}, \mathbf{v}_h), \quad \forall \Psi_h \in \mathcal{X}_h.
\end{equation}

\item \textbf{Subspace update}: Construct the updated augmented subspace $\mathcal{X}_{H,h}^{(\ell+1)} = \mathcal{X}_H + \text{span}\{\widehat{\Phi}_h^{(\ell+1)}\}$ and solve the eigenvalue problem: Find $(\lambda_h^{(\ell+1)}, \Phi_h^{(\ell+1)}) \in \mathbb{R} \times \mathcal{X}_{H,h}^{(\ell+1)}$ satisfying $\|\mathbf{u}_h^{(\ell+1)}\|_{0} = 1$ and
\begin{equation}\label{Aug_Eigenvalue_Problem_k_mix}
A(\Phi_h^{(\ell+1)}, \Psi_{H,h}) = \lambda_h^{(\ell+1)} m(\mathbf{u}_h^{(\ell+1)}, \mathbf{v}_{H,h}), \quad \forall \Psi_{H,h} \in \mathcal{X}_{H,h}^{(\ell+1)}.
\end{equation}

\item \textbf{Iteration}: Set $\ell = \ell + 1$ and return to Step 2 until convergence.
\end{enumerate}
\end{algorithm}

\begin{lemma}\label{lem:Hh}
Let $\widehat{\Phi}_h^{(\ell+1)}$ be the solution of linear boundary value problem \eqref{Linear_Equation_k_elastic}. For the augmented subspace $\mathcal{X}_{H,h}^{(\ell+1)} := \mathcal{X}_H + \text{span}\{\widehat{\Phi}_h^{(\ell+1)}\}$, there exists a bounded linear operator $\Pi_{H,h}: \mathcal{X} \rightarrow \mathcal{X}_{H,h}^{(\ell+1)}$ such that
\begin{enumerate}
\item[(1)] $A(\Psi - \Pi_{H,h}\Psi, \Psi_{H,h}) = 0, \quad \forall \Psi_{H,h} \in \mathcal{X}_{H,h}^{(\ell+1)}$,
\item[(2)] $\|\Pi_{H,h}\Psi\|_{\mathcal{X}} \leq C\| \Psi\|_{\mathcal{X}}, \quad \forall \Psi \in \mathcal{X},$
\end{enumerate}
where the constant $C$ is independent of the Lam\'{e} constant $\underline{\lambda}$.
\end{lemma}
\begin{proof}
Since $\widehat{\Phi}_{h}^{(\ell+1)}$ is the solution to linear boundary value problem \eqref{Linear_Equation_k_elastic}, it follows that $\widehat{\Phi}_{h}^{(\ell+1)} \in \mathcal{X}_h$. Let $\widehat{\Phi}_H$ denote the finite element projection of $\widehat{\Phi}_{h}^{(\ell+1)}$ onto $\mathcal{X}_H$. Define $\widehat{\Phi}_h := \widehat{\Phi}_{h}^{(\ell+1)} - \widehat{\Phi}_H \in \mathcal{X}$. By the results of Lin et al. \cite{LinXieXu}, $\widehat{\Phi}_h$ is non-zero. Normalize $\widehat{\Phi}_h$ to obtain $\tilde{\Phi}_h = \widehat{\Phi}_h  /  \|\widehat{\Phi}_h\|_{\mathcal{X}}$.

For any $\Psi \in\mathcal{X}$, let $\Pi_{H,h}\Psi \in \mathcal{X}_{H,h}$. Assume that $\Pi_{H,h}\Psi = \Psi_H + \gamma \tilde{\Phi}_h$, where $\Psi_H \in \mathcal{X}_H$ and $\gamma\in\mathbb{R}$. To construct the bounded linear operator $\Pi_{H,h}$, it suffices to determine $\Psi_H$ and $\gamma$ for any given $\Psi \in\mathcal{X}$.

Any element in the augmented subspace $\mathcal{X}_{H,h}$ can be expressed as $\Phi_H + \alpha \tilde{\Phi}_h$, where $\Phi_H \in\mathcal{X}_H$ and $\alpha\in\mathbb{R}$. Thus
\begin{eqnarray}\label{eq:equal}
&&A\left(\Psi-\Pi_{H,h} \Psi, \Psi_{H,h} \right)=0, \ \forall \Psi_{H,h} \in \mathcal{X}_{H,h}, \nonumber \\
&\Leftrightarrow&A\left(\Psi-\Psi_H - \gamma \tilde{\Phi}_h, \Phi_H + \alpha \tilde{\Phi}_h \right)=0, \ \forall \Phi_H \in\mathcal{X}_H,\ \alpha\in\mathbb{R},  \nonumber \\
&\Leftrightarrow& \begin{cases}
A(\Psi - \Psi_H, \Phi_H) - \gamma A(\tilde{\Phi}_h, \Phi_H) = 0, \ \forall \Phi_H \in\mathcal{X}_H, \\	
A(\Psi - \Psi_H, \tilde{\Phi}_h) - \gamma A(\tilde{\Phi}_h, \tilde{\Phi}_h) = 0,
\end{cases} \nonumber \\
&\Leftrightarrow& \begin{cases}
A(\Psi_H, \Phi_H) + \gamma A(\tilde{\Phi}_h, \Phi_H) = A(\Psi, \Phi_H), \ \forall \Phi_H \in\mathcal{X}_H, \\	
A(\Psi_H, \tilde{\Phi}_h) + \gamma A(\tilde{\Phi}_h, \tilde{\Phi}_h) = A(\Psi, \tilde{\Phi}_h).
\end{cases}
\end{eqnarray}

Given that $\widehat{\Phi}_H$ is the finite element projection of $\widehat{\Phi}_{h}^{(\ell+1)}$ onto $\mathcal{X}_H$ and $\widehat{\Phi}_h = \widehat{\Phi}_{h}^{(\ell+1)} - \widehat{\Phi}_H \in \mathcal{X}$, the following orthogonality holds for any $\Theta_{H} \in \mathcal{X}_{H}$,
\begin{align}\label{eq:A0}
0 = A(\widehat{\Phi}_h, \Theta_{H}) = A(\tilde{\Phi}_h, \Theta_{H}).
\end{align}

Let $\{\Phi_{j,H}\}_{j=1}^{N_H}$ be a basis for the finite-dimensional space $\mathcal{X}_H$. For $\Psi_H = \sum_{j=1}^{N_h} t_j \Phi_{j,H}$, substituting \eqref{eq:A0} into \eqref{eq:equal} yields the discrete system
\begin{equation}\label{eq:Axb}
\begin{pmatrix}
\mathbf{A} & \mathbf{0} \\ \mathbf{0}^T & \xi
\end{pmatrix}\begin{pmatrix}
\mathbf{t} \\  \gamma
\end{pmatrix} = \begin{pmatrix}
\mathbf{r} \\ \zeta
\end{pmatrix}
\end{equation}
where $\mathbf{A} = \left( A(\Phi_{j,H}, \Phi_{k,H}) \right)_{N_H\times N_H}$, $\mathbf{t} = (t_1, t_2, \cdots, t_{N_H})^T$, $\mathbf{r} = \left( A\left(\Psi,\Phi_{j,H}\right) \right)_{N_H\times 1}$, $\mathbf{0}$ is the $N_h$-dimensional zero vector, $\xi = A(\tilde{\Phi}_h, \tilde{\Phi}_h)$ and $\zeta = A(\Psi, \tilde{\Phi}_h)$.

By Theorems \ref{thm:Wellposedness} and \ref{thm:errorEstimate}, the matrix $\mathbf{A}$ is invertible, and the linear system \eqref{eq:Axb} admits a bounded solution. Thus, for any $\Psi\in\mathcal{X}$, the system uniquely determines $\Psi_H$ and $\gamma$, thereby defining the bounded linear operator $\Pi_{H,h}$.
This completes the proof of this lemma.
\end{proof}

By virtue of Theorem \ref{thm:Wellposedness}, Theorem \ref{thm:errorEstimate}, Lemma \ref{lem:Hh} and equation \eqref{eq:AubinNitsche}, we conclude that $\eta_{\mathcal{X}}(\mathcal{X}_H) \leq CH$. Furthermore, the error estimates in Theorem \ref{thm:EigErr}, Corollary \ref{cor:Eigerr} and Corollary \ref{cor:Eigerr1} remain valid for the augmented subspace $\mathcal{X}_{H,h}$.

\begin{theorem}\label{Error_Estimate_One_Smoothing_Theorem}
For $\ell \geq 1$, let $(\lambda_h^{(\ell)}, \Phi_h^{(\ell)}) \in \mathbb{R} \times \mathcal{X}_h$ be an approximate eigenpair, and define the spectral projection operator $E_h^{(\ell)}: \mathcal{X} \rightarrow \text{span}\{\Phi_h^{(\ell)}\}$ as
\begin{equation}
A(E_h^{(\ell)}\Psi, \Phi_h^{(\ell)}) = A(\Psi, \Phi_h^{(\ell)}), \quad \forall \Psi \in \mathcal{X}.
\end{equation}
Then, the approximate eigenpair $(\lambda_h^{(\ell+1)}, \Phi_h^{(\ell+1)}) \in \mathbb{R} \times \mathcal{X}_h$ obtained via Algorithm \ref{Alg:Algorithm_1} satisfies
\begin{align}
\|\Phi - \Phi_h^{(\ell+1)}\|_{\mathcal{X}} &\leq C\| \Phi - \mathcal{P}_h\Phi\|_{\mathcal{X}} \| \Phi - \Phi_h^{(\ell)}\|_{\mathcal{X}} + C\| \Phi - \mathcal{P}_h\Phi\|_{\mathcal{X}}, \label{Estimate_h_k_1_a} \\
\|\lambda \mathbf{u} - \lambda_h^{(\ell+1)} \mathbf{u}_h^{(\ell+1)}\|_{0} &\leq C \eta_{\mathcal{X}}(\mathcal{X}_{H,h}^{(\ell+1)}) \| (I - \mathcal{P}_h)\Phi\|_{\mathcal{X}}
+ C\|\Phi - \Phi_h^{(\ell+1)}\|_{\mathcal{X}}^2. \label{Estimate_h_k_1_b}
\end{align}
\end{theorem}

\begin{proof}
From equations \eqref{eq:MixedEigenForm}, \eqref{eq:Energy_Projection} and \eqref{Linear_Equation_k_elastic}, for any $\Psi_h \in \mathcal{X}_{h}$, we have
$$A(\mathcal{P}_{h} \Phi - \widehat{\Phi}_h^{(\ell+1)}, \Psi_h) =A(\Phi- \widehat{\Phi}_h^{(\ell+1)}, \Psi_h) = m(\lambda \mathbf{u} - \lambda_{h}^{(\ell)} \mathbf{u}_{h}^{(\ell)}, \mathbf{v}_h).$$
By Theorem \ref{thm:Wellposedness}, we derive
\begin{equation}\label{thmeq1}
\|\mathcal{P}_{h} \Phi - \widehat{\Phi}_h^{(\ell+1)}\|_{\mathcal{X}} \leq C \|\lambda \mathbf{u} - \lambda_{h}^{(\ell)} \mathbf{u}_{h}^{(\ell)}\|_{0}.
\end{equation}

The discrete eigenvalue problem \eqref{parallel_correct_eig_exact_1} can be regarded as a subspace approximation to the eigenvalue problem \eqref{eq:EigenValueProblem}. Combining \eqref{eq:corerr1}, \eqref{eq:corerr2} and \eqref{thmeq1}, we obtain the following estimates
\begin{align}\label{thmeq2}
&\|\mathcal{P}_{h} \Phi - \widehat{\Phi}_h^{(\ell+1)}\|_{\mathcal{X}} \leq C  \eta_{\mathcal{X}}\left(\mathcal{X}_{h}\right)\|\left(I-\mathcal{P}_{h}\right) \Phi\|_{\mathcal{X}} + C \|\Phi -\Phi_{h}^{(\ell)}\|_{\mathcal{X}}^2 \nonumber\\
&\leq C  \eta_{\mathcal{X}}\left(\mathcal{X}_{h}\right)\|\left(I-\mathcal{P}_{h}\right) \Phi\|_{\mathcal{X}}
+ C \|\left(I-\mathcal{P}_{h}\right) \Phi\|_{\mathcal{X}} \|\Phi -\Phi_{h}^{(\ell)}\|_{\mathcal{X}}.
\end{align}

Similarly, the discrete eigenvalue problem \eqref{Aug_Eigenvalue_Problem_k_mix} can be viewed as a subspace approximation to \eqref{eq:EigenValueProblem}. Therefore, using \eqref{eq:corerr1}, \eqref{eq:corerr2}, \eqref{Aug_Eigenvalue_Problem_k_mix}, \eqref{thmeq2}, Lemma \ref{lem:AubinNitsche}, Theorem \ref{thm:EigErr} and Corollary \ref{cor:Eigerr1}, we derive
\begin{align*}\label{Error_u_u_h_2}
&\|\Phi-\Phi_h^{(\ell+1)}\|_{\mathcal{X}} \leq C \inf_{\Psi_{H,h}\in \mathcal{X}_{H,h}^{(\ell+1)}}\|\Phi-\Psi_{H,h}\|_{\mathcal{X}} \nonumber\\
&\leq C \|\Phi-\widehat{\Phi}_h^{(\ell+1)}\|_{\mathcal{X}}
\leq C \big(\|\Phi- \mathcal{P}_h\Phi\|_{\mathcal{X}} + \| \mathcal{P}_h\Phi-\widehat{\Phi}_h^{(\ell+1)}\|_{\mathcal{X}} \big)\nonumber\\
&\leq C  \eta_{\mathcal{X}}\left(\mathcal{X}_{h}\right)\|\left(I-\mathcal{P}_{h}\right) \Phi\|_{\mathcal{X}} + C \|\Phi -\Phi_{h}^{(\ell)}\|_{\mathcal{X}}^2 + C \|\Phi- \mathcal{P}_h\Phi\|_{\mathcal{X}} \\
&\leq C\|\Phi- \mathcal{P}_h\Phi\|_{\mathcal{X}} \|\Phi -\Phi_{h}^{(\ell)}\|_{\mathcal{X}} + C \|\Phi- \mathcal{P}_h\Phi\|_{\mathcal{X}},
\end{align*}
and
\begin{eqnarray*}\label{Error_u_u_h_2_Negative}
\|\lambda \mathbf{u} - \lambda_h^{(\ell+1)} \mathbf{u}_h^{(\ell+1)}\|_{0}
\leq C \eta_{\mathcal{X}}\left(\mathcal{X}_{H,h}^{(\ell+1)}\right)\|\left(I-\mathcal{P}_{h}\right) \Phi\|_{\mathcal{X}} + C \|\Phi - \Phi_{i,h}^{(\ell + 1)}\|_{\mathcal{X}}^2,
\end{eqnarray*}
which completes the proof of desired results.
\end{proof}

\section{Numerical experiments}\label{section5}
In this section, we provide several numerical tests to support our theoretical findings. It should be noted that we use the high performance computers of State Key Laboratory of Mathematical Sciences (SKLMS), Chinese Academy of Sciences, to power the computations done here. Two 18-core 2.3 GHz Intel Xeon Gold 6140 processors and 192 GB of RAM are provided in each computer node.

\subsection{Numerical tests for boundary value problem}
This subsection presents two numerical examples to validate the theoretical results of Mini mixed finite element discretization for linear elasticity boundary value problem.
\subsubsection{Numerical test for convergence order}\label{exm:elastic1}
Consider the mixed formulation for linear elasticity equation \eqref{eq:ElasticStrongForm} on the square domain $\Omega = (0,1)^2$ with the homogeneous Dirichlet boundary condition for $\mathbf{u}$. The parameters we take are $\underline{\mu} = 1$ and $\underline{\lambda} = 1$. The pressure variable $p := (\lambda + \mu) \nabla\cdot \mathbf{u}$ and the right-hand side term $\mathbf{f}$ are selected according to the analytical solution
\begin{align}
\mathbf{u} = \begin{pmatrix}
\sin(\pi x) \sin(\pi y) \\ (x^2 - x)(y^2 - y)
\end{pmatrix}.
\end{align}
For the Mini mixed finite element discretization, we use a quasi-uniform triangular mesh $\mathcal{T}_h$. The linear equations formed by the coefficient matrix from \eqref{eq:MixDiscreteWeakFormGeneral} are solved by the direct solver in the package PETSc \cite{petsc-web-page,petsc-user-ref,petsc-efficient}.

Table \ref{tab:example1} lists the errors and convergence rates for displacement $\mathbf{u}$ in the $H^1$-norm and pressure $p$ in the $L^2$-norm. From Table \ref{tab:example1}, we can find that the error convergence order for $\mathbf{u}$ is about $1$, which is optimal according to the conclusion in Corollary \ref{corollary2.1}. And the error convergence order for $p$ is about $1.5$, demonstrating the superconvergence phenomenon.
\begin{table}[htbp]
\centering
\caption{Errors and convergence rates in Section \ref{exm:elastic1}.}
\label{tab:example1}
\begin{tabular}{ccccc}
\toprule
$h$       & $\|\mathbf{u} - \mathbf{u}_h\|_1$ & Rate & $\| p - p_h\|_0$ & Rate  \\
\midrule
1/16    & 0.2079010 &  -   & 0.03518050 & -  \\
1/32    & 0.1038420 & 1.0015 & 0.01091750 & 1.6881 \\
1/64    & 0.0518874 & 1.0009 & 0.00354690 & 1.6220 \\
1/128   & 0.0259345 & 1.0005 & 0.00119236 & 1.5727 \\
1/256   & 0.0129649 & 1.0003 & 0.00040977 & 1.5409 \\
\bottomrule
\end{tabular}
\smallskip
\end{table}

\subsubsection{Numerical test for locking-free property}\label{exm2}
In this example, we are still concerned with the mixed formulation for linear elasticity equation \eqref{eq:ElasticStrongForm} on the square domain $\Omega = (0,1)^2$ with the homogeneous Dirichlet boundary condition for $\mathbf{u}$. We set $\underline{\mu} = 1$ and test a wide range of near-incompressible parameter: $\underline{\lambda} = 1, 10, 10^3, 10^6$. The progression of $\underline{\lambda}$ spans from a compressible material to an extremely near-incompressible scenario. The pressure $p := (\lambda + \mu) \nabla\cdot \mathbf{u}$ and the right-hand side term $\mathbf{f}$ are selected according to the analytical solution
\begin{align}\label{5.2}
\mathbf{u} = \begin{pmatrix}
\sin(2\pi y) (-1 + \cos(2\pi x)) + \frac{1}{\underline{\mu} + \underline{\lambda}} \sin(\pi x)\sin(\pi y) \\
\sin(2\pi x) (1 - \cos(2\pi y)) + \frac{1}{\underline{\mu} + \underline{\lambda}} \sin(\pi x)\sin(\pi y)
\end{pmatrix}.
\end{align}
From \eqref{5.2}, we can obtain $\nabla\cdot\mathbf{u}\rightarrow 0$ as $\underline{\lambda} \rightarrow\infty$. A quasi-uniform triangular mesh is utilized for the Mini mixed finite element method, and the coefficient matrix from \eqref{eq:MixDiscreteWeakFormGeneral} are calculated by the direct solver in the package PETSc \cite{petsc-web-page,petsc-user-ref,petsc-efficient}.

Table \ref{tab:example_merged_vertical} presents the errors and convergence rates for displacement $\mathbf{u}$ in the $H^1$-norm and pressure $p$ in the $L^2$-norm. As can be seen from Table \ref{tab:example_merged_vertical}, the errors vary within a small range as $\underline{\lambda}$ increases significantly, which indicates the locking-free property of our method.
\begin{table}[!htbp]
\centering
\caption{Errors and convergence rates with different $\underline{\lambda}$ in Section \ref{exm2}.}
\label{tab:example_merged_vertical}
\begin{tabular}{cccccc}
\toprule
$\underline{\lambda}$ & $h$       & $\|\mathbf{u}-\mathbf{u}_h\|_1$ & Rate & $\|p-p_h\|_0$ & Rate  \\
\midrule
\multirow{5}{*}{$1$} & 1/16      & 1.3408800 & -    & 0.14301900 & - \\
     & 1/32      & 0.6741410 & 0.9921 & 0.04714470 & 1.6010 \\
     & 1/64      & 0.3374610 & 0.9983 & 0.01602380 & 1.5569 \\
     & 1/128     & 0.1687610 & 0.9997 & 0.00554838 & 1.5301 \\
     & 1/256     & 0.0843801 & 1.0000 & 0.00194064 & 1.5155 \\
\midrule
\multirow{5}{*}{$10$} & 1/16      & 1.3400500 & -    & 0.28715200 & - \\
     & 1/32      & 0.6717910 & 0.9962 & 0.09529370 & 1.5914 \\
     & 1/64      & 0.3358860 & 1.0000 & 0.03278520 & 1.5393 \\
     & 1/128     & 0.1678840 & 1.0005 & 0.01144620 & 1.5182 \\
     & 1/256     & 0.0839201 & 1.0004 & 0.00402172 & 1.5090 \\
\midrule
\multirow{5}{*}{$10^3$} & 1/16     & 1.3462800 & -    & 0.39597600 & - \\
       & 1/32     & 0.6730920 & 1.0001 & 0.13208100 & 1.5840 \\
       & 1/64     & 0.3361610 & 1.0017 & 0.04574870 & 1.5296 \\
       & 1/128    & 0.1679340 & 1.0013 & 0.01604210 & 1.5119 \\
       & 1/256    & 0.0839243 & 1.0007 & 0.00564924 & 1.5057 \\
\midrule
\multirow{5}{*}{$10^6$} & 1/16     & 1.3464000 & -    & 0.39779400 & - \\
       & 1/32     & 0.6731180 & 1.0002 & 0.13269700 & 1.5839 \\
       & 1/64     & 0.3361670 & 1.0017 & 0.04596580 & 1.5295 \\
       & 1/128    & 0.1679360 & 1.0013 & 0.01611920 & 1.5118 \\
       & 1/256    & 0.0839247 & 1.0007 & 0.00567657 & 1.5057 \\
\bottomrule
\end{tabular}
\end{table}

\subsection{Numerical tests for eigenvalue problem}\label{section5.2}
In this test, we concentrate on the mixed formulation for linear elasticity eigenvalue problem on the two-dimensional domain $\Omega=(0,1)^2$ with the homogeneous Dirichlet boundary condition for $\mathbf{u}$. We choose $\underline{\mu} = 1$ and $\underline{\lambda} = 10, 10^3, 10^6$. A quasi-uniform triangular mesh is used. For the confirmation of locking-free of our method, we compare the Mini mixed finite element method with the linear conforming element and Crouzeix-Raviart (CR) element, which are utilized to directly solve the linear elasticity eigenvalue problem corresponding to \eqref{eq:ElasticStrongForm}. In \cite{MR4251006}, the authors illustrate that CR element is locking-free for this case. We adopt the Krylov-Schur algorithm from SLEPc \cite{SLEPc} to solve the related algebraic eigenvalue problem.

Table \ref{table_n} shows that the minimum and fourth eigenvalues obtained by the Mini mixed element, linear conforming element and CR element. As can be observed from Table \ref{table_n}, the eigenvalues obtained by the linear conforming element become large as $\underline{\lambda}$ increases, which states that the numerical eigenvalues deteriorate as $\underline{\lambda} \rightarrow\infty$. While the eigenvalues obtained by the Mini mixed element and CR element remain stable as $\underline{\lambda}$ increases, which indicates that the proposed Mini mixed finite element method for solving the linear elasticity eigenvalue problem is locking-free, like CR element.
\begin{table}[!htbp]
\centering
\caption{The minimum and fourth eigenvalues obtained by Mini mixed element, linear conforming element and CR element.}
\label{table_n}
\begin{tabular}{cccccccc}
\toprule
$\underline{\lambda}$  & $h$ & $\lambda_{1,h}^{Mini}$ & $\lambda_{4,h}^{Mini}$ & $\lambda_{1,h}^{P_1}$ & $\lambda_{4,h}^{P_1}$ & $\lambda_{1,h}^{CR}$ & $\lambda_{4,h}^{CR}$  \\
\midrule
$10$   & 1/12    & 54.19961627 & 137.9264142 & 61.90769312 & 147.4860118 & 51.31894655 & 119.6042608 \\
$10$   & 1/24    & 52.60299740 & 128.3538448 & 54.77453045 & 138.7446263 & 51.88046195 & 123.7058478 \\
$10$   & 1/48    & 52.20836253 & 125.9336849 & 52.77706265 & 129.8678713 & 52.02789033 & 124.7698539 \\
$10$   & 1/96    & 52.11038053 & 125.3303801 & 52.25452783 & 126.3562255 & 52.06532506 & 125.0396697 \\
$10$   & 1/192   & 52.08597665 & 125.1800164 & 52.12215008 & 125.4393635 & 52.07472346 & 125.1074028 \\
\midrule
$10^3$ & 1/12    & 54.59322328 & 142.5489024 & 530.2551921 & 1994.617321 & 51.53263777 & 122.6704945 \\
$10^3$ & 1/24    & 52.89417557 & 131.7622770 & 181.0086738 & 571.1743764 & 52.12934706 & 126.7441161 \\
$10^3$ & 1/48    & 52.47826635 & 129.0678401 & 90.29026469 & 257.4211242 & 52.28762822 & 127.8148957 \\
$10^3$ & 1/96    & 52.37549306 & 128.4002599 & 63.75275685 & 179.7577017 & 52.32795927 & 128.0877989	 \\
$10^3$ & 1/192   & 52.34995917 & 128.2343872 & 55.57534545 & 148.6296729 & 52.33809513 & 128.1564085 \\
\midrule
$10^6$ & 1/12    & 54.59823957 & 142.5968282 & 456471.1975 & 1871927.886 & 51.53521661 & 122.7004734 \\
$10^6$ & 1/24    & 52.89775696 & 131.7964139 & 111812.6353 & 453808.9601 & 52.13236886 & 126.7740862 \\
$10^6$ & 1/48    & 52.48155874 & 129.0989761 & 27859.71646 & 112597.9815 & 52.29078829 & 127.8450483 \\
$10^6$ & 1/96    & 52.37872180 & 128.4307131 & 7011.435218 & 28179.98138 & 52.33115808 & 128.1180173 \\
$10^6$ & 1/192   & 52.35317313 & 128.2646787 & 1807.790492 & 7131.280925 & 52.34143824 & 128.1882994 \\
\bottomrule
\end{tabular}
\smallskip
\end{table}

\subsection{Numerical verification of augmented subspace method}
This section systematically validates the algebraic error estimates for Algorithm \ref{Alg:Algorithm_1}, including the dependence of errors on the coarse mesh size $H$ and the locking-free property of our proposed method.

Consider the mixed formulation for linear elasticity eigenvalue problem on the domain $\Omega=(0,1)^2$ with the homogeneous Dirichlet boundary condition for $\mathbf{u}$. We choose $\underline{\mu} = 1$ and $\underline{\lambda} = 10, 10^3, 10^6$. $\mathcal{X}_H$ and $\mathcal{X}_h$ are chosen as the Mini mixed finite element spaces defined on the coarse mesh $\mathcal{T}_H$ and the fine mesh $\mathcal{T}_h$, respectively. By regularly refining the uniform triangular mesh, the fine mesh $\mathcal{T}_h$ is generated from the coarse mesh $\mathcal{T}_H$. The coarse mesh sizes are $H = \sqrt{2}/12$, $\sqrt{2}/24$ and $\sqrt{2}/48$, and the fine mesh size is fixed as $h = \sqrt{2}/192$. The initial eigenfunction approximation is produced by solving the mixed eigenvalue problem on the coarse space $\mathcal{X}_H$. Then we do the iteration steps by the augmented subspace method defined by Algorithm \ref{Alg:Algorithm_1}. The package PETSc \cite{petsc-web-page,petsc-user-ref,petsc-efficient} employs the direct solver to solve the linear equation \eqref{Linear_Equation_k_elastic} in Algorithm \ref{Alg:Algorithm_1}. And the Krylov-Schur algorithm from SLEPc \cite{SLEPc} solves the eigenvalue problems \eqref{parallel_correct_eig_exact_1} and \eqref{Aug_Eigenvalue_Problem_k_mix} in Algorithm \ref{Alg:Algorithm_1}. It should be emphasized that the exact mixed finite element eigenfunction is obtained by directly solving the mixed eigenvalue problem on the fine space $\mathcal{X}_h$. To make it more intuitive, the notations in all the following figures with and without the ``dir" superscript represent the augmented subspace approximations and the exact mixed finite element eigenfunctions, respectively.

\subsubsection{The minimum eigenvalue corresponds to the convergence of eigenfunction}
Firstly, we examine the numerical errors corresponding to Mini mixed finite element space $\mathcal{X}_H$ with various sizes $H$ in order to determine how the mesh size $H$ affects the convergence rate. Figures \ref{la1_1th}, \ref{la3_1th} and \ref{la6_1th} display the convergence behaviors for the first eigenfunction using the augmented subspace techniques corresponding to the coarse mesh sizes $H = \sqrt{2}/12, \sqrt{2}/24, \sqrt{2}/48$ and $\underline{\lambda} = 10, 10^3, 10^6$, respectively. Let us take Figure \ref{la1_1th} as an example. Norms $\|\cdot\|_{0}$ and $\|\cdot\|_{1}$ for $\mathbf{u}$ and $\|\cdot\|_0$ for $p$ have similar convergence rates of $0.081136$, $0.020989$, $0.0051387$. This finding demonstrates the second order convergence speed with respect to $H$ of the augmented subspace approach specified by Algorithm \ref{Alg:Algorithm_1}. Additionally, from Figure \ref{la1_1th}, we can find that the smaller the mesh size $H$ of the coarse mesh $\mathcal{T}_H$ is, the faster the convergence speed of Algorithm \ref{Alg:Algorithm_1} will be.

From Figures \ref{la1_1th}, \ref{la3_1th} and \ref{la6_1th} in total, as the Lam\'e constant $\underline{\lambda}$ increases from $10$ to $10^3$ and then to $10^6$, the errors and their convergence rates remain basically unchanged. This verifies that Algorithm \ref{Alg:Algorithm_1} is stable, i.e., locking-free.

\begin{figure}[!http]
\centering
\includegraphics[width=15.0cm,height=4.0cm]{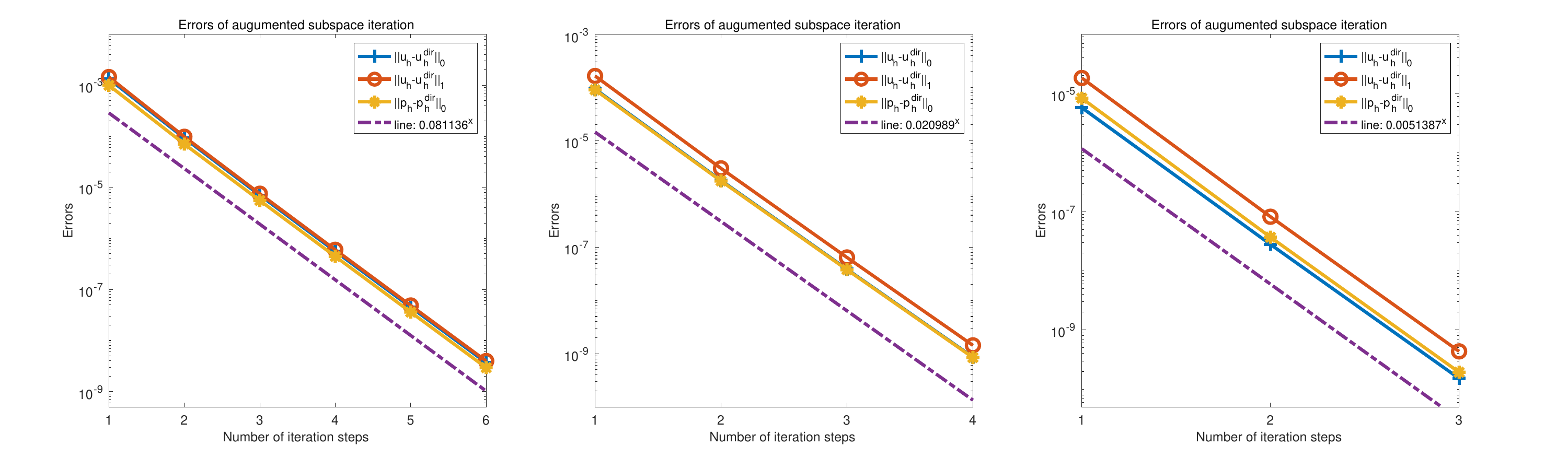}
\caption{The convergence behaviors for the first eigenfunction by Algorithm \ref{Alg:Algorithm_1} corresponding to Lam\'e constant $\underline{\lambda} = 10$ and the coarse mesh sizes $H=\sqrt{2} / 12$, $\sqrt{2} / 24$ and $\sqrt{2} / 48$.}\label{la1_1th}
\end{figure}

\begin{figure}[http!]
\centering
\includegraphics[width=15.0cm,height=4.0cm]{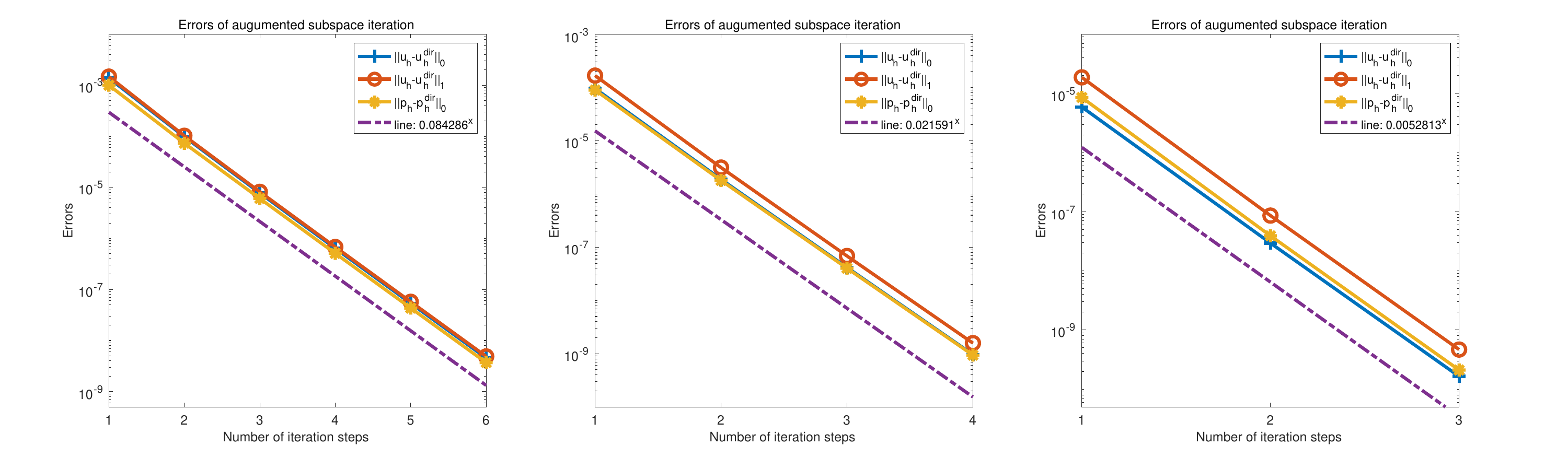}
\caption{The convergence behaviors for the first eigenfunction by Algorithm \ref{Alg:Algorithm_1} corresponding to Lam\'e constant $\underline{\lambda} = 10^3$ and the coarse mesh sizes $H=\sqrt{2} / 12$, $\sqrt{2} / 24$ and $\sqrt{2} / 48$.}\label{la3_1th}
\end{figure}

\begin{figure}[http!]
\centering
\includegraphics[width=15.0cm,height=4.0cm]{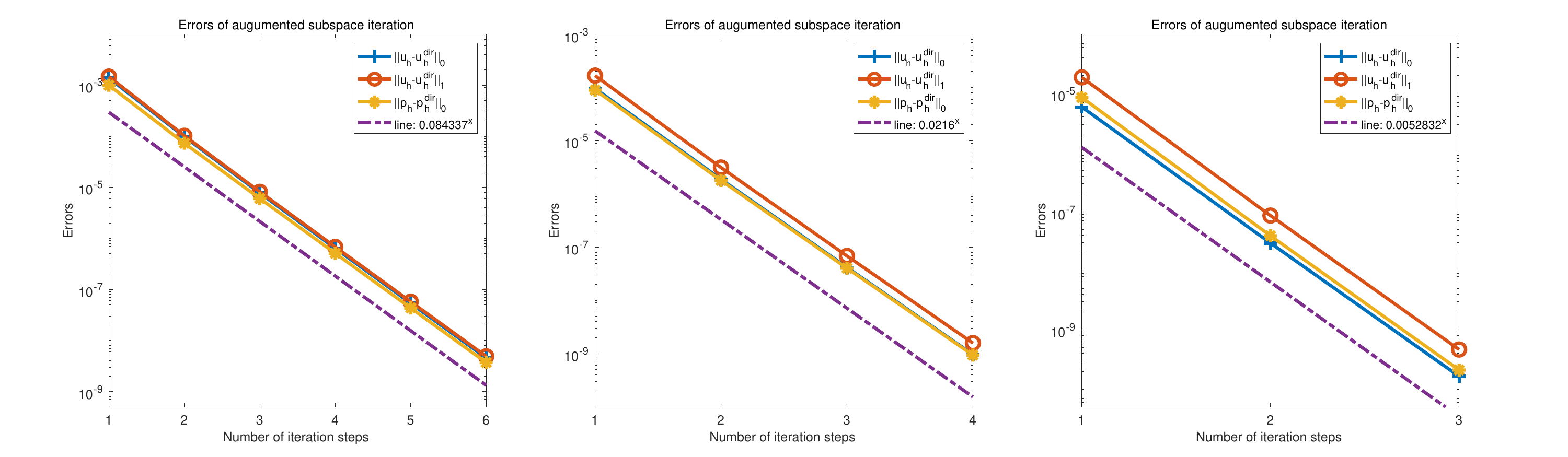}
\caption{The convergence behaviors for the first eigenfunction by Algorithm \ref{Alg:Algorithm_1} corresponding to Lam\'e constant $\underline{\lambda} = 10^6$ and the coarse mesh sizes $H=\sqrt{2} / 12$, $\sqrt{2} / 24$ and $\sqrt{2} / 48$.}\label{la6_1th}
\end{figure}

\subsubsection{The first four eigenvalues correspond to the convergence of eigenfunction}
In this subsection, we check the performance of Algorithm \ref{Alg:Algorithm_1} for computing the smallest $4$ eigenpairs. The respective convergence behaviors for the smallest $4$ eigenfunctions by Algorithm \ref{Alg:Algorithm_1} are displayed in Figure \ref{la1_4}, \ref{la3_4} and \ref{la6_4} with $\underline{\lambda} = 10, 10^3, 10^6$. The coarse mesh sizes are adopted as $H=\sqrt{2} / 12$, $\sqrt{2} / 24$ and $\sqrt{2} / 48$. Using the $4$-th eigenfunction in Figure \ref{la1_4} as an example, we can find that the corresponding convergence rates are $0.22751$, $0.066263$ and $0.016969$, which indicates the second order convergence speed with respect to $H$ of the method defined by Algorithm \ref{Alg:Algorithm_1}. It can also be observed from Figure \ref{la1_4} that the convergence speed of Algorithm \ref{Alg:Algorithm_1} increases as the mesh size $H$ of the coarse mesh $\mathcal{T}_H$ decreases. And the $4$-th eigenfunction's convergence rate is slower that the $1$-st eigenfunction's, illustrating that the eigenvalue is positively correlated with the convergence rate, as the first eigenvalue is smaller than the fourth one.

From an overall view of Figures \ref{la1_4}, \ref{la3_4} and \ref{la6_4}, when the Lam\'e constant $\underline{\lambda}$ is raised from $10$ to $10^3$ and further to $10^6$, the errors and their convergence rates stay almost unchanged. This serves as verification that Algorithm \ref{Alg:Algorithm_1} is locking-free.

\begin{figure}[http!]
\centering
\includegraphics[width=15.0cm,height=4.0cm]{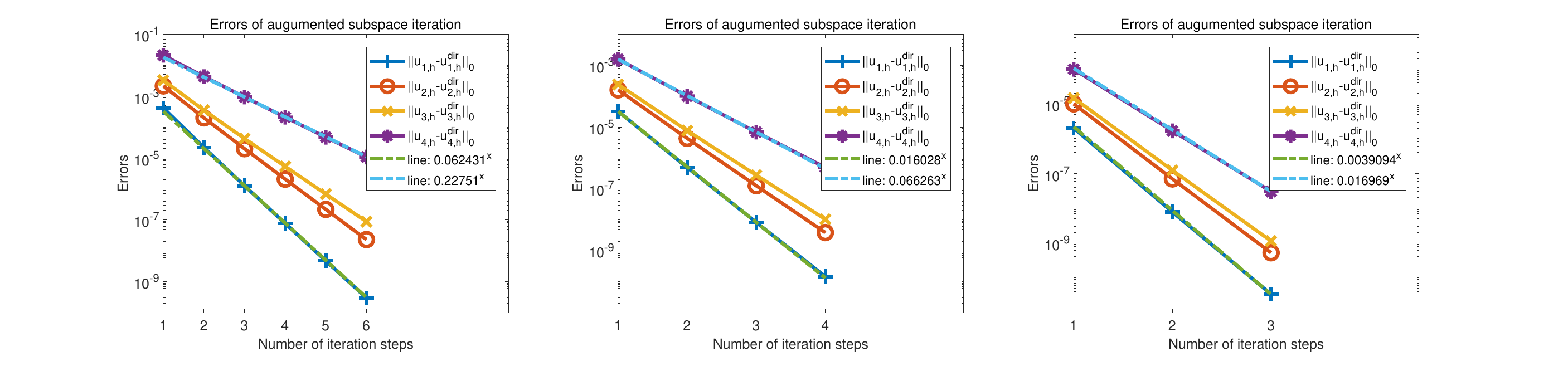}
\caption{The convergence behaviors for the smallest 4 eigenfunctions by Algorithm \ref{Alg:Algorithm_1} corresponding to Lam\'e constant $\underline{\lambda} = 10$ and the coarse mesh size $H=\sqrt{2} / 12$, $\sqrt{2} / 24$ and $\sqrt{2} / 48$.}\label{la1_4}
\end{figure}

\begin{figure}[http!]
\centering
\includegraphics[width=15.0cm,height=4.0cm]{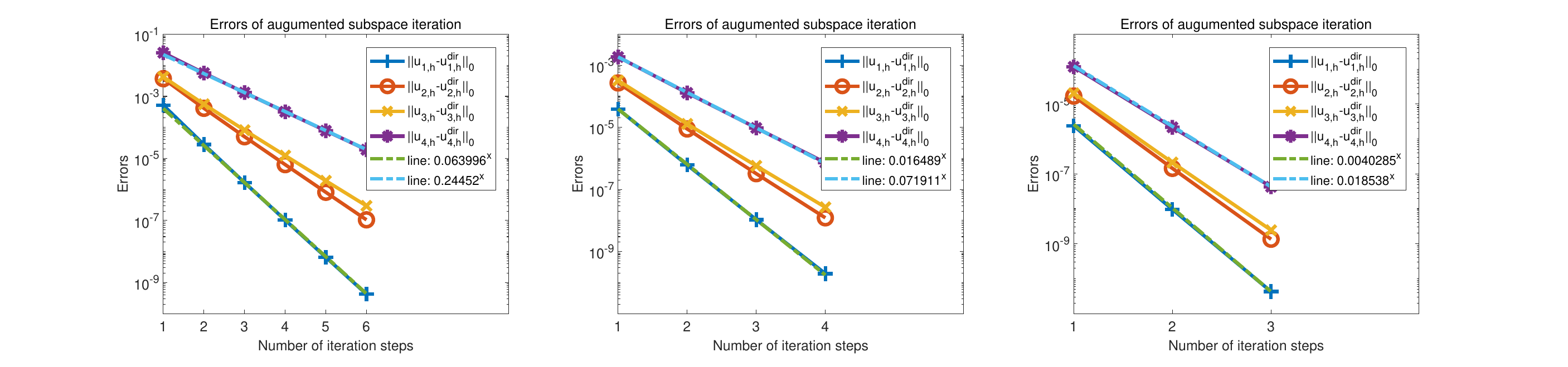}
\caption{The convergence behaviors for the smallest 4 eigenfunctions by Algorithm \ref{Alg:Algorithm_1} corresponding to Lam\'e constant $\underline{\lambda} = 10^3$ and the coarse mesh size $H=\sqrt{2} / 12$, $\sqrt{2} / 24$ and $\sqrt{2} / 48$.}\label{la3_4}
\end{figure}

\begin{figure}[http!]
\centering
\includegraphics[width=15.0cm,height=4.0cm]{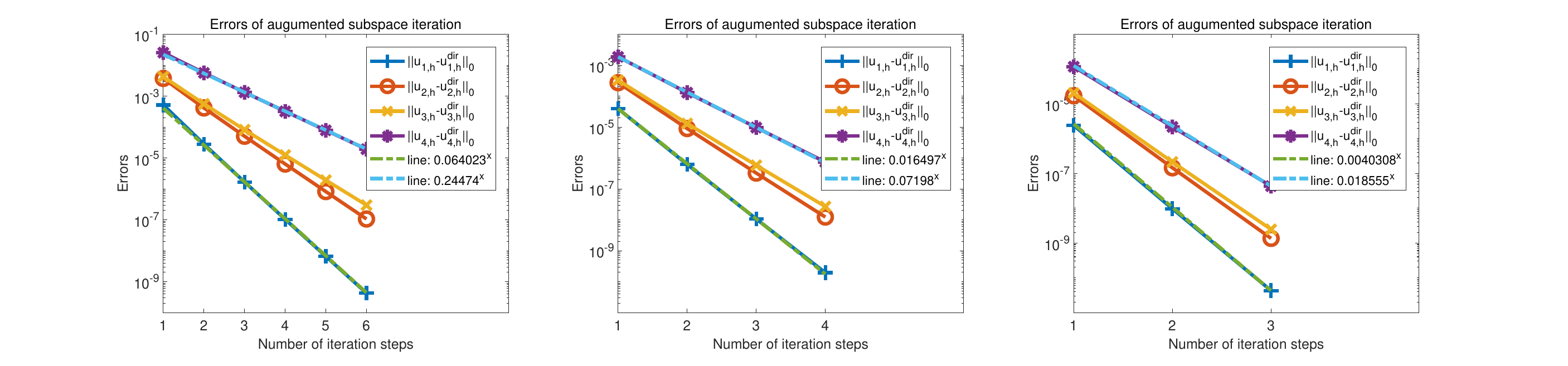}
\caption{The convergence behaviors for the smallest 4 eigenfunctions by Algorithm \ref{Alg:Algorithm_1} corresponding to Lam\'e constant $\underline{\lambda} = 10^6$ and the coarse mesh size $H=\sqrt{2} / 12$, $\sqrt{2} / 24$ and $\sqrt{2} / 48$.}\label{la6_4}
\end{figure}

\subsubsection{The fourth eigenvalue corresponds to the convergence of eigenfunction}
The final subsection is to access the performance of augmented subspace method for computing the only $4$-th eigenpair. Figures \ref{la1_4th}, \ref{la3_4th} and \ref{la6_4th} show the corresponding convergence behaviors for the only 4-th eigenfunction by Algorithm \ref{Alg:Algorithm_1} with the coarse mesh sizes $H=\sqrt{2} / 12, \sqrt{2} / 24, \sqrt{2} / 48$ and $\underline{\lambda}=10, 10^3, 10^6$. Take Figure \ref{la1_4th} as an example. The convergence rates corresponding to $\|\cdot\|_{0}$ and $\|\cdot\|_{1}$ for $\mathbf{u}$ and $\|\cdot\|_0$ for $p$ are $0.22988$, $0.065701$ and $0.015915$, separately. This finding provides that the augmented subspace technique provided by Algorithm \ref{Alg:Algorithm_1} has the second order convergence speed with respect to $H$. Figure \ref{la1_4th} also shows that the smaller the mesh size $H$ of the coarse mesh $\mathcal{T}_H$ is, the faster Algorithm \ref{Alg:Algorithm_1} converges.

Based on Figures \ref{la1_4th}, \ref{la3_4th} and \ref{la6_4th} together, regardless of the increase of Lam\'e constant $\underline{\lambda}$ from $10$ to $10^3$ and subsequently to $10^6$, the errors and their convergence rates remain largely unchanged. This demonstrates that Algorithm \ref{Alg:Algorithm_1} is locking-free.

\begin{figure}[http!]
\centering
\includegraphics[width=15.0cm,height=4.0cm]{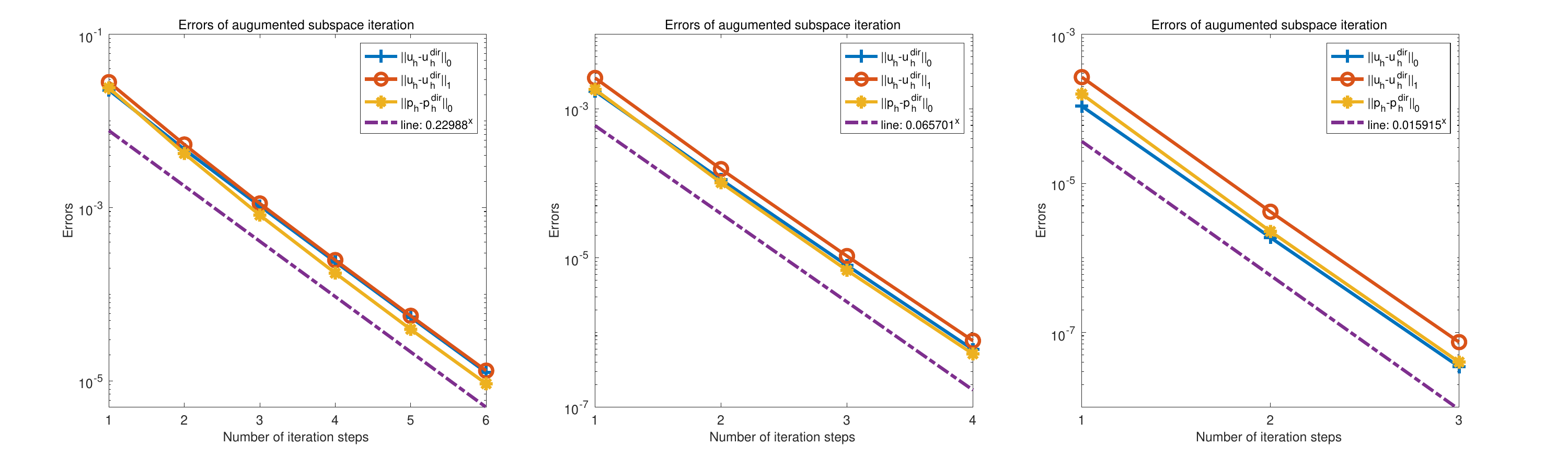}
\caption{The convergence behaviors for the only 4-th eigenfunction by Algorithm \ref{Alg:Algorithm_1} corresponding to Lam\'e constant $\underline{\lambda} = 10$ and the coarse mesh sizes $H=\sqrt{2} / 12$, $\sqrt{2} / 24$ and $\sqrt{2} / 48$.}\label{la1_4th}
\end{figure}

\begin{figure}[http!]
\centering
\includegraphics[width=15.0cm,height=4.0cm]{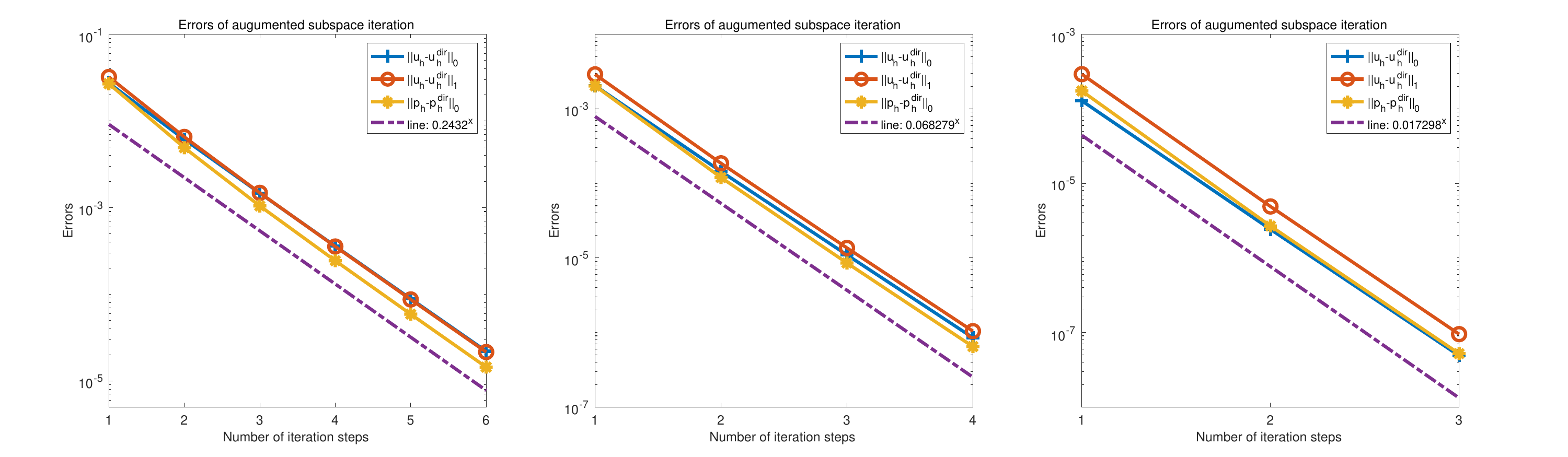}
\caption{The convergence behaviors for the only 4-th eigenfunction by Algorithm \ref{Alg:Algorithm_1} corresponding to Lam\'e constant $\underline{\lambda} = 10^3$ and the coarse mesh sizes $H=\sqrt{2} / 12$, $\sqrt{2} / 24$ and $\sqrt{2} / 48$.}\label{la3_4th}
\end{figure}

\begin{figure}[http!]
\centering
\includegraphics[width=15.0cm,height=4.0cm]{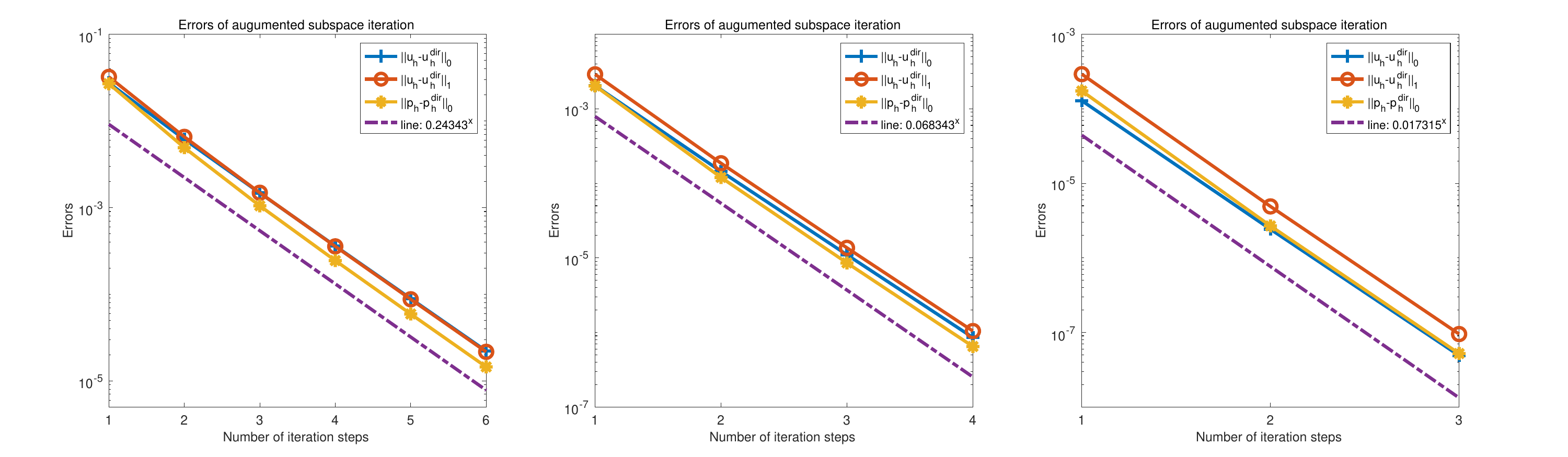}
\caption{The convergence behaviors for the only 4-th eigenfunction by Algorithm \ref{Alg:Algorithm_1} corresponding to Lam\'e constant $\underline{\lambda} = 10^6$ and the coarse mesh sizes $H=\sqrt{2} / 12$, $\sqrt{2} / 24$ and $\sqrt{2} / 48$.}\label{la6_4th}
\end{figure}

\section{Concluding remarks}\label{section6}
This paper studies the Mini mixed finite element method for the almost incompressible linear elasticity problems, including the source and eigenvalue problems. The situation without locking is given particular consideration. At the theoretical level, the error estimates independent of the material parameters are derived, and the locking-free property of discrete solution of Mini mixed finite element is verified through the fruitful numerical experiments. At the algorithmic level, as an efficient algorithm for solving eigenvalue problems, the augmented subspace method under the framework of mixed finite element is proposed and analyzed. The numerical experiments verify the algebraic error order and the locking-free property of the algorithm. In the future, we will further develop the algorithm to handle problems such as transmission-type eigenvalue problems.

\section*{Acknowledgements}
This work was supported by the Strategic Priority Research Program of
Chinese Academy of Sciences (XDB0620203, XDB0640000, XDB0640300),
National Key Research and Development Program of China (2023YFB3309104),
National Natural Science Foundation of China (12331015, 12301475, 12301465),
Science Challenge Project (TZ2025007),
National Key Laboratory of Computational Physics (6142A05230501),
State Key Laboratory of Mathematical Sciences,
National Center for Mathematics and Interdisciplinary Science, Chinese Academy of Sciences,
and the Research Foundation for Beijing University of Technology New Faculty (006000514122516).

\section*{Declarations}

\subsection*{Conflict of interest}
All authors declare that they have no conflict of interest.

\subsection*{Data availability}
All data generated or analysed during the current study are available from the corresponding author on reasonable request.


\end{document}